\documentclass[12pt,twoside,final,psamsfonts]{amsart}

\usepackage[psamsfonts]{amssymb}
\usepackage{times,a4wide}
\usepackage[ ]{showkeys}
\usepackage{color}

\theoremstyle{plain}
\newtheorem*{theorem*}{Theorem}

\newtheorem*{CoronaN}{Corona Theorem for $N$}
\newtheorem{theorem}{Theorem}[section]
\newtheorem{proposition}[theorem]{Proposition}
\newtheorem*{proposition*}{Proposition}
\newtheorem{corollary}[theorem]{Corollary}
\newtheorem*{corollary*}{Corollary}
\newtheorem{lemma}[theorem]{Lemma}
\newtheorem*{lemma*}{Lemma}

\theoremstyle{definition}

\newtheorem*{remark*}{Remark}

\theoremstyle{definition}

\newtheorem*{definition*}{Definition}
\newcommand{\nc}{\newcommand}

\newcommand{\N}{{\mathbb N}}
\newcommand{\D}{{\mathbb D}}

\newcommand{\mD}{\mathcal{D}}

\newcommand{\Int}{\operatorname{Int}}

\newcommand{\Har}{\operatorname{Har}}

\newcommand{\dist}{\operatorname{dist}}

\newcommand{\Lra}{\Longrightarrow}
\newcommand{\lmto}{\longmapsto}
\newcommand{\eps}{\varepsilon}

\nc{\bea}{\begin{eqnarray}}
\nc{\eea}{\end{eqnarray}}
\nc{\beqa}{\begin{eqnarray*}}
\nc{\eeqa}{\end{eqnarray*}}
\nc{\Hi}{H^{\infty}}
\nc{\loi}{\ell^{\infty}}
\nc{\NL}{N^+\vert \Lambda}
\nc{\hf}{{\mathcal H}_{\phi}}
\nc{\liL}{\lambda\in\Lambda}
\nc{\nn}{\nonumber}
\nc{\ol}{\overline}
\nc{\jr}{_{j,r}}
\nc{\ir}{_{i,r}}
\nc{\kr}{_{k,r}}
\nc{\lr}{_{l,r}}
\nc{\dst}{\displaystyle}

\newenvironment{proof*}{\vskip 2mm\noindent {}}{$\blacksquare$ \vskip 2mm}
\numberwithin{equation}{section}

\renewcommand{\Re}{\operatorname{Re}}

\renewcommand{\qedsymbol}{$\blacksquare$}

\title
{Finitely generated Ideals in the Nevanlinna class}

\author[A. Hartmann, X. Massaneda, A. Nicolau]{Andreas Hartmann,
Xavier Massaneda, Artur Nicolau}

\address{A.Hartmann:   Universit\'e de Bordeaux\\
 IMB\\ 351 cours de la Lib\'era\-tion\\ 33405 Talence\\ France}
 \email{Andreas.Hartmann@math.u-bordeaux1.fr}  

\address{X.Massaneda: Universitat  de Barcelona\\
Departament de Matem\`a\-tiques i Inform\`atica \& BGS-Math\\
Gran Via 585, 08007-Bar\-ce\-lo\-na\\ Spain}
\email{xavier.massaneda@ub.edu}

\address{A.Nicolau: Universitat Aut\`onoma de Barcelona\\
Departament de Matem\`a\-tiques \\
Edifici C, 08193-Bellaterra\\ Spain}
\email{artur@mat.uab.cat}

\date{\today}



\thanks{Second and third authors supported by the Generalitat de Catalunya (grants 2014SGR 289 and 2014SGR 75) and the Spanish Ministerio de Ciencia e Innovaci\'on (projects MTM2014-51834-P and  MTM2014-51824-P)}

\begin{document}
\maketitle

\begin{abstract}
In this paper we investigate finitely generated ideals in the Nevanlinna class. We prove analogues to some known results for the algebra of bounded analytic functions $H^\infty$. We also show that, in contrast to the $H^\infty$ case, the stable rank of the Nevanlinna class is strictly bigger than 1. 
\end{abstract}

\section{Introduction}

The aim of this paper is to investigate analogues for the Nevanlinna class of some known results on finitely generated ideals of the algebra $H^{\infty}$ of bounded analytic functions in the unit disk $\D$, equipped with the supremum norm $\|f\|_{\infty}=\sup\limits_{z\in\D}|f(z)|$. 

Let us begin by recalling these results.
The first one concerns interpolating sequences. A sequence of points $\Lambda=\{\lambda_n\}_{n\in\N}$  in $\D$ is called
\emph{interpolating for} $H^{\infty}$ if for every bounded sequence $\{w_n\}_{n\in\N}$ of complex numbers, there exists a function $f\in H^{\infty}$
such that
$
 f(\lambda_n)=w_n, n\in\N.
$
By a famous result by Carleson \cite{Carl} a sequence $\{\lambda_n\}_n$ is interpolating for $H^{\infty}$ if and only if 
\[
 \inf_{n\in\mathbb N} \prod_{k\neq n}\left|\frac{\lambda_k-\lambda_n}{1-\overline{\lambda_n}\lambda_k}\right|>0 .
\]
A Blaschke product {with simple zeros is called an}
\emph{interpolating Blaschke product} 
{if its zeros are an interpolating sequence}.

The next important result in the context of this paper is Carleson's corona theorem:  every family $\{f_1,\ldots,f_m\}$ of functions in $H^{\infty}$  satisfying
\[
 \inf_{z\in \D}\sum_{i=1}^m |f_i(z)| >0
\]
generates the whole algebra.  {See \cite{Garn} or \cite{NikTr}.} More generally, we denote by $I_{H^\infty}(f_1,\ldots, f_m)$ the ideal generated by the 
functions $f_1,\ldots,f_m$ in $H^{\infty}$.
The general structure of these ideals is not well understood (see the references  {\cite{Bourgain}, \cite{DGM}, \cite{GIM}-\cite{GMN}, \cite{Mort1}, \cite{Mort2}, \cite{Tol83}, \cite{Tol88}} for more information).
As it turns out, in certain situations the ideals can be characterized by
growth conditions. More precisely, the following ideals have been studied:
\[ 
 J_{H^\infty}(f_1,\ldots,f_m)=\bigl\{f\in H^{\infty}:\, \exists c= c(f)>0\, ,\, |f(z)|\le c\sum_{i=1}^m |f_i(z)|\, ,\, z\in \D\bigr\}.
 \]
It is obvious that $I_{H^\infty}(f_1,\ldots,f_m)\subset J_{H^\infty}(f_1,\ldots,f_m)$. This leads us to the third circle of results we are interested in here.
Tolokonnikov \cite{Tol83} proved that the following conditions are equivalent:
\begin{itemize}
\item [(a)] $ J_{H^\infty}(f_1,\ldots,f_m)$ contains an interpolating Blaschke product, 
\item [(b)] $I_{H^\infty} (f_1,\ldots,f_m)$ contains an interpolating Blaschke product, 
\item [(c)] $\inf\limits_{z\in\D} \sum_{i=1}^m (|f_i(z)|+(1-|z|^2)|f_i'(z)|)>0$.
\end{itemize}


As it turns out, in the special situation of two generators with no common zeros
these conditions are equivalent to $ I_{H^\infty}(f_1,f_2)=J_{H^\infty}(f_1,f_2)$.
In the case of two generators $f_1$ and $f_2$ with common zeros, we have 
$I(f_1,f_2) = J(f_1,f_2)$ if and only if $I(f_1,f_2)$ contains a 
function of the form $BC$ where $B$ is an interpolating Blaschke product and $C$ is the Blaschke product formed with the common zeros of $f_1$ and $f_2$  (see \cite{GMN}).

Let us now turn to the framework we want to discuss
in this paper. We are interested in analogues of the above results
for the \emph{Nevanlinna class} $N$, consisting of the holomorphic functions $f$ on $\D$ such that
$\log_+|f|$ has a positive harmonic majorant on $\D$. Equivalently, $f\in N$ if and only if $f$ is holomorphic on $\D$ and 
\[
 \lim_{r\to 1^-} \int_{\partial\D} \log_+|f(r\zeta)| d\sigma(\zeta)<\infty\ .
\]
Here $d\sigma$ denotes the normalized Lebesgue measure on the unit circle. 

As a general rule we shall see that the results for $H^\infty$ translate to the Nevanlinna setting provided that the boundedness of the elements described above is replaced by a control given by a positive harmonic majorant (or minorant). Let $\Har_+(\D)$ be the cone of positive harmonic functions in the unit disk $\D$. Recall that any $H\in \Har_+(\D)$ is the Poisson integral of a positive measure $\mu$ on the unit circle, that is 
\[
 H(z)=P[\mu](z)=\int_{\partial\D} P(z,\zeta) d\mu(\zeta),
\]
where
\[
 P(z,\zeta)=\Re\left(\frac{\zeta+z}{\zeta-z}\right)=\frac{1-|z|^2}{|\zeta-z|^2}
\]
is the Poisson kernel in $\D$.

It is a standard fact that functions $f$ in the Nevanlinna class admit non-tangential limits $f^*$ at almost every point of the circle. 
It is also well-known that any $f\in N$ can be factored as
$
 f=BSE,
$
where $B$ is a Blaschke product containing the zeros of $f$, $S$ is a singular inner function and $E$ is the outer function:
\[
 E(z)= C\exp\left\{\int_{\partial\D} \frac{\zeta+z}{\zeta-z} \log|f^*(\zeta)| d\sigma(\zeta)\right\},
\]
where $|C|=1$. In particular
\[
 \log |E(z)|= P[\log|f^*|](z),\quad z\in\D\ .
\]
A function $S$ is singular inner if there exists a positive measure $\mu$ on $\partial\D$ singular with respect to the Lebesgue measure such that
\[
 S(z)=\exp\left\{- \int_{\partial\D} \frac{\zeta+z}{\zeta-z} d\mu(\zeta) \right\},\quad z\in\D\ .
\]

For the Nevanlinna class R. Mortini observed that a well known result of T. Wolff implies the following corona theorem (see \cite{Mort0} or  \cite{Mart}).

\begin{CoronaN}[R. Mortini]
Let $I(f_1,\ldots,f_m)$ denote the ideal generated in $N$ by a given family of functions 
$f_1,\ldots, f_m\in N$. Then $I(f_1,\ldots,f_m)=N$ if and only if there exists $H\in \Har_+(\D)$
such that
\[
 \sum_{i=1}^m |f_i(z)|\ge e^{-H(z)},\quad z\in \D.
\]
\end{CoronaN}

We need to define the ideal corresponding to $J_{H^{\infty}}$ in $N$.
This will be done in the following way:
\[ 
 J(f_1,\ldots,f_m)=\bigl\{f\in N:\exists H=H(f)\in \Har_+(\D)\, ,\, |f(z)|\le e^{H(z)}\sum_{i=1}^m |f_i(z)|\, ,\, z\in \D\bigr\}.
\]
It is clear that $I(f_1,\ldots,f_m)\subset J(f_1,\ldots,f_m)$. Let us also mention that, by the  previous corona theorem, in the case
when $J(f_1,\ldots,f_m)=N$, then $I(f_1,\ldots,f_m)=N$. 

Recall that a sequence space is called {\it ideal} if it is stable with respect
to pointwise multiplication by bounded sequences. For the following definition see also \cite{HMNT}.
\begin{definition*}
A sequence of points $\Lambda=\{\lambda_n\}_n$ in $\D$ is called \emph{interpolating for} $N$ (denoted $\Lambda\in\Int N$) if the trace space $N|\Lambda$ is ideal.
\end{definition*}
Equivalently, $\Lambda\in \Int N$ if for every bounded sequence $\{v_n \}_n$ of complex numbers there exists $f\in N$ such that 
\[
 f(\lambda_n)=v_n, \quad n\in\N.
\]
Interpolating sequences for the Nevanlinna class were first investigated by Naftalevitch \cite{Naft}
starting from an {\it a priori} fixed target space which forces interpolating sequences to be
confined in a finite union of Stolz angles. 

A rather complete study, based on the above definition, was carried out much later in \cite{HMNT}. In particular, 
it was proved that a sequence $\{\lambda_n\}_n$ is interpolating for $N$ if and only if there exists a  positive harmonic function $H \in \Har_+(\D)$ such that 
\bea\label{NIS}
 \prod_{k:k\neq n}\left|\frac{\lambda_k-\lambda_n}{1-\overline{\lambda_n}\lambda_k}\right| \geq e^{- H(\lambda_n)} \, , n \in \N . 
\eea
Moreover, it was also shown that if $\Lambda\in \Int N$, then the trace space is given by
\bea\label{TrSpace}
 N|\Lambda=\left\{\{w_n\}_n:\exists H\in \Har_+(\D)\ ,\, \log_+|w_n| \le H(\lambda_n)\right\}.
\eea
It was also noticed that in the previous condition only the factors corresponding to $\lambda_k$ close to $\lambda_n$ are relevant. More precisely, fixed any $c\in (0,1)$, the condition
\begin{equation}\label{interpol-local}
\prod\limits_{\stackrel{k:k\neq n}{\rho(\lambda_k,\lambda_n)\leq c}} \left |\frac{\lambda_k-\lambda_n}{1-\overline{\lambda_n}\lambda_k}\right|\geq e^{-H(\lambda_n)},\quad n\in\N
\end{equation}
is sufficient for $\Lambda$ to be interpolating (see \cite[Proposition 4.1]{HMNT}).

A Blaschke product the zeros of which forms an interpolating sequence for the Nevanlinna class is
called a \emph{Nevanlinna interpolating Blaschke product}.

The analogues of the results mentioned above in the context of $H^\infty$ 
read as follows.

\begin{theorem}\label{MainThm}
Let $f_1,\ldots,f_m$ be functions in $N$. Then the following conditions are equivalent:
\begin{itemize}
\item  [(a)] $I(f_1,\ldots,f_m)$ contains a Nevanlinna interpolating Blaschke product,  
\item  [(b)]$J(f_1,\ldots,f_m)$ contains a Nevanlinna interpolating Blaschke product, 
\item  [(c)] There exists a function $H\in \Har_+(\D)$ such that 
\[
\sum_{i=1}^m (|f_i(z)|+(1-|z|^2)|f_i'(z)|)\ge e^{-H(z)}\, , \ z\in \D.
\]
\end{itemize}
In case $m=2$, if $f_1$ and $f_2$ have no common zeros, the above conditions are equivalent to 

\begin{itemize}
\item [(d)] $I(f_1,f_2)=J(f_1,f_2)$.
\end{itemize}
\end{theorem}

As in $H^{\infty}$ each of the conditions (a)-(c) implies $I(f_1,\ldots, f_m)=J(f_1,\ldots,f_m)$. 
However, when $m\ge 3$, the converse fails, as will be explained after the proof of the result. Also, like in the $H^{\infty}$-situation, if the two generators $f_1$ and $f_2$ have common zeros, then $I(f_1,f_2)=J(f_1,f_2)$ if and only if $I(f_1,f_2)$ contains a function of the form $BC$ where $B$ is a Nevanlinna interpolating Blaschke product and $C$ is the Blaschke product formed with the common zeros of $f_1$ and $f_2$. 

Our proof of Theorem \ref{MainThm}
uses some of the ideas from both \cite{Tol83} and \cite{GMN}, but also some specific 
properties of the Nevanlinna class. In particular we will make use of a new description
of Nevanlinna interpolating sequences in terms of harmonic measure, which we discuss now.

Denote by
\[
 \rho(z,w)=\left|\frac{z-w}{1-\bar z w}\right|
\]
the pseudohyperbolic distance in $\D$, and by $D(z,r)=\{w\in\D : \rho(z,w)<r\}$ the corresponding disk of center $z$ and radius $r\in(0,1)$. Let $B$ denote the Blaschke product with zeros $\Lambda =\{\lambda_n \}_n$ and let 
\[
 b_{\lambda_n}(z)=\frac{\overline{\lambda}_n}{|\lambda_n|}
 \frac{\lambda_n-z}{1-\bar\lambda_n z}\ ,\quad B_n(z)=\frac{B(z)}{b_{\lambda_n}(z)}\ .
\]
In these terms $B(z)=\prod_n b_{\lambda_n}(z)$ and $|b_{\lambda_n}(z)|=\rho(z,\lambda_n)$. Given $H\in\Har_+(\D)$, consider the disks 
$\mD_n^H=D(\lambda_n, e^{-H(\lambda_n)})$ and the domain
\[
 \Omega_n^H=\D\setminus\bigcup_{\stackrel{k:k\neq n}{\rho(\lambda_k,\lambda_n)\le 1/2}}
  \mD_k^H.
\]
It will be clear from the proof of Theorem~\ref{IntThm} below that the choice of the constant 1/2 in the definition of $\Omega_n^H$ is of no relevance; it can be replaced by any $c\in (0,1)$. Let $\omega(z, E,\Omega)$ denote the harmonic measure at $z\in\Omega$ of
the set $E\subset \partial \Omega$ in the domain $\Omega$.
The following result collects several new descriptions of Nevanlinna interpolating sequences which will be used in the proof of Theorem \ref{MainThm}.

\begin{theorem}\label{IntThm}
Let $\Lambda=\{\lambda_n\}_n$ be a Blaschke sequence of distinct points in $\D$ and let $B$ be the Blaschke
product with zero set $\Lambda$. The following statements are equivalent:
\begin{itemize}
\item [(a)] $\Lambda$ is an interpolating sequence for $N$, that is, there exists $ H\in \Har_+(\D)$ such that
\begin{equation*}\label{interpolation}
 (1-|\lambda_n|^2)|B'(\lambda_n)|=|B_n(\lambda_n)|\ge e^{-H(\lambda_n)},\quad n\in\N.
\end{equation*}
\item [(b)]  There exists $ H\in \Har_+(\D)$ such that $|B(z)|\ge e^{-H(z)}\rho(z,\Lambda)$, $z\in\D$,
\item [(c)] There exists $ H\in \Har_+(\D)$ such that $|B(z)|+(1-|z|^2)|B'(z)|\ge e^{-H(z)}$, $z\in\D$,
\item [(d)] There exists $ H\in \Har_+(\D)$ such that the disks $ \mD_k^H$ 
are pairwise disjoint, and 
\[
 \inf_{n\in\mathbb N} \omega(\lambda_n,\partial\D,\Omega_n^H)>0.
\]
\end{itemize}
\end{theorem}

Statement (d) and its proof are modelled after the corresponding version for $H^{\infty}$, proved in \cite{GGJ}. {Descriptions of interpolating and sampling sequences in Bergman spaces in terms of harmonic measure can be found in \cite{OS}.}  It will be clear from the proof that (d) can be replaced by a seemingly stronger statement: for every $\varepsilon\in(0,1)$ there exists $ H\in \Har_+(\D)$ such that the disks $ \mD_k^H$ 
are pairwise disjoint, and 
\[
 \inf_{n\in\mathbb N} \omega(\lambda_n,\partial\D,\Omega_n^H)\geq 1-\varepsilon.
\]

The paper is organized as follows. In the next section we shall prove Theorem \ref{IntThm}
and some Corollaries which will be used later.  Section
\ref{S3} is devoted to the the equivalence of the statements (a), (b) and (c) of Theorem \ref{MainThm}  and Section \ref{S3b} to condition (d)
in  the case $m=2$. At the end of Section 4 it is also explained that when $m>2$ then condition (d) does not imply any of the previous ones. The last Section is devoted to present two related open problems. The first one concerns the stable rank of $N$ and the second is a version of the well known $f^2$ problem of T. Wolff (see \cite[p. 319]{Garn}), solved by S. Treil in the context of $H^{\infty}$ \cite{Tr2}.

A final word about notation. Throughout the paper  $A\lesssim B$ will mean that there is an absolute constant $C$ such that $A \le CB$, and we write $A\asymp B$  if both $A\lesssim B$ and $B\lesssim A$.

It is a pleasure to thank Raymond Mortini for drawing our attention to the Corona Theorem in the Nevanlinna class and to his paper \cite{Mort0}. 


\section{Interpolating sequences in the Nevanlinna class}\label{S2}

We start with an elementary lemma.

\begin{lemma}\label{bounded-f}
Let $f\in H^{\infty}$ with $\|f\|_{\infty} = \sup\limits_{z \in \D} |f(z)| \le 1$.
\begin{itemize}

\item [(a)] For all $z,\lambda\in \D$,
\[
 |f(w)-f(\lambda)|\leq 2 \rho(w,\lambda)\ .
\]

\item [(b)] Fix $0<\delta <1/5$. 
If $|f(z)|\le \delta^4$ and $(1-|z|^2)|f'(z)|\ge \delta$ for a fixed $z\in \D$, then
$|f(w)|\ge \delta^4$ if $\rho(z,w)=\delta^2$.

\item [(c)] If $\rho(z,w)\le 1/2$, then
\[
 \big|(1-|z|^2)f'(z)-(1-|w|^2)f'(w)\big|\le 6\, \rho(z,w).
\]
\end{itemize}
\end{lemma}

\begin{proof}[Proof of the Lemma \ref{bounded-f}.]

(a) This is a direct consequence of Schwarz' Lemma:
\[
 \rho(f(w),f(\lambda))\leq \rho(w,\lambda)\qquad z,\lambda\in\D .
\]

(b) Assume first that $z=0$ and write $ f(w)=f(0)+f'(0)w+w^2g(w)$.
Since $\|f\|_{\infty}\le 1$ and $|f'(0)| \leq 1$, 
we have $|g(w)|\le 3$ for every $w \in \D$, and hence, 
\[
 |f(w)|\geq |f'(0)||w|-|f(0)|-3|w|^2\ , w \in \D . 
\]
Since $\delta\leq 1/5$, then for $|w|=\delta^2$ we have 
$
 |f(w)|\geq \delta^3-\delta^4-3\delta^4\geq \delta^4,
$
as desired.

For arbitrary $z\in \D$ we apply the previous argument to the function $f\circ \phi_z$, where 
\[
 \phi_z(w)=\frac{z-w}{1-\bar z w}
\]
is the holomorphic automorphism of $\D$ exchanging $0$ and $z$. 
Since $|(f\circ \phi_z)'(0)|=(1-|z|^2)|f'(z)|\geq \delta$ and $|(f\circ \phi_z)(0)|=|f(z)|\le\delta^4$, 
taking $\zeta \in \D$ such that $w= \phi_z (\zeta)$, we get  $|f(w)|=|(f\circ \phi_z)(\zeta)|\geq\delta^4$ if  $\rho(z,w)=|\zeta|=\delta^2$.

\bigskip

(c) Again, assume first that $z=0$. If $|\zeta|\leq 1/2$ then
\begin{align*}
 \left|f'(0)-(1-|\zeta|^2)f'(\zeta)\right|&\leq \left|f'(0)-f'(\zeta)\right| + |\zeta|^2 |f'(\zeta)|
  \leq |f'(0)-f'(\zeta)|+\frac{|\zeta|^2}{1-|\zeta|^2}\ .
\end{align*}
Let $g(\zeta)=f'(\zeta)-f'(0)$. For $|\zeta|\leq 1/2$ we have
\[
 |g(\zeta)|=|f'(\zeta)-f'(0)|\leq\frac 1{1-|\zeta|^2}+1\leq \frac 73 .
\]
Applying (a) to $h(z):=3/7\, g(z/2)$ we deduce that 
\[
 |g(\zeta)|\leq \frac{14}3 |\zeta|,\quad |\zeta|\leq 1/2\ .
\]
Finally, if $|\zeta|\leq 1/2$, from the above estimate we deduce that
\[
  \left|f'(0)-(1-|\zeta|^2)f'(\zeta)\right| \leq \frac{14}3 |\zeta|+\frac{|\zeta|^2}{1-|\zeta|^2}\leq \frac{16}3 |\zeta|\leq 6|\zeta|,
\]
as desired.

For general $z\in\D$  we use the case $z=0$ and the invariance by automorphisms of $\tilde\nabla f(z)=(1-|z|^2)f'(z)$ , that is,  
$\tilde\nabla (f\circ \phi_z)(\zeta)=(\tilde\nabla f)(\phi_z(\zeta))$ for any $\zeta , z \in \D$. Then, for $|\zeta|\leq 1/2$,
\begin{align*}
\left|(1-|z|^2)f'(z)-(1-|\zeta|^2)(f\circ\phi_z)'(\zeta) \right|&=\left|(f\circ\phi_z)'(0)-(1-|\zeta|^2) (f\circ\phi_z)'(\zeta)\right|\leq 6|\zeta|\ .
\end{align*}
Letting $\zeta=\phi_z(w)$ and using the invariance we see that
$(1-|\zeta|^2)(f\circ\phi_z)'(\zeta) = (1-|w|^2)f'(w) $
and the result follows.
\end{proof}

In the proofs we will repeatedly use the well-known \emph{Harnack inequalities}: for $H\in\Har_+(\mathbb D)$ and $z,w\in\mathbb \D$,
\begin{eqnarray}\label{Harnack}
 \frac{1-\rho(z,w)}{1+\rho(z,w)}\leq\frac{H(z)}{H(w)}\leq\frac{1+\rho(z,w)}{1-\rho(z,w)}\ .
\end{eqnarray}
In certain parts of this paper, we will need to suppose that $z,w$ are pseudohyperbolically close:
$\rho(z,w)<x$ for some $0<x<1$, so that $(x-1)/(x+1)\le H(z)/H(w) \le (x+1)/(x-1)$. The constant
$(x+1)/(x-1)$ will occasionally be called the Harnack constant. 

In this section we shall always assume, without loss of generality, that positive harmonic functions $H\in\Har_+(\mathbb D)$ defining pseudohyperbolic neighborhoods $D(\lambda,e^{-H(\lambda)})$ are big enough so that the corresponding
Harnack constant is at most 2. More specifically, let $H\in\Har_+(\mathbb D)$ be such that $H(z) \geq \ln3$ for any $z \in \D$; then 
\begin{eqnarray}\label{Harnack1}
 \frac{H(w)}{2} \leq H(z) \leq 2H(w) \quad \text{if} \quad  \rho(z,w) \leq e^{-H(z)} .
\end{eqnarray}


\bigskip



Here is another easy and useful fact.
\begin{lemma}\label{Nevfcts}
There exists a universal constant $C>0$ such that for any $f\in N$ with $|f(z)|\le e^{H(z)}$, $z\in\D$, for some $H\in \Har_+(\D)$, one has 
\begin{itemize}
\item[(a)] For every $z\in\D$, $(1-|z|)|f'(z)|\le e^{CH(z)}$, $(1-|z|)^2|f''(z)|\le e^{CH(z)}$.
\item[(b)] For every $z,w\in\D$ with $\rho(z,w)\le 1/3$, $|f(z)-f(w)|\le \rho(z,w)e^{CH(z)}$.
\end{itemize}
\end{lemma}

\begin{proof}
The estimates in (a) are an easy consequence of Cauchy's formula and Harnack's inequality. The estimate in (b) follows immediately from (a) integrating $f'$ from $z$ to $w$ and using again
Harnack's inequality.
\end{proof}

\begin{proof}[Proof of Theorem \ref{IntThm}.]


(a)$\Longrightarrow$ (b).  By hypothesis there exists $H_0\in\Har_+(\D)$ satisfying Theorem~\ref{IntThm}(a), and therefore the 
disks $\mathcal{D}_n^{H_0}=D(\lambda_n, e^{-H_0(\lambda_n)})$ are pairwise disjoint. 
We will show that condition (b) holds with $H= C H_0$, where $C$ is an absolute constant. Consider the disks $\mathcal{D}_n^{2H_0} = D(\lambda_n , e^{-2 H_0 (\lambda_n)})$.  

i) Pick $z\in \mathcal{D}_n^{2 H_0}$. 
By construction, $\lambda_n$ is the closest point of $\Lambda$ to $z$
 and
\[
|B(z)|=|B_n(z)||b_{\lambda_n}(z)|=|B_n(z)| \rho(z,\Lambda)
\]
Since $B_n$ does not vanish in $\mathcal{D}_n^{2H_0}$, 
by Harnack's inequalities \eqref{Harnack} and \eqref{Harnack1}, there exists an absolute constant $C>0$ such that 
\[
  |B_n(z)|\geq |B_n(\lambda_n)|^C \geq e^{- C H_0(\lambda_n)}\ge  e^{-2C H_0(z)} \ .
\]

ii) Let  $\Omega:=\D\setminus \cup_n \mathcal D_n^{2H_0}$. The function $B$ is holomorphic and 
non-vanishing in $\Omega$. Let $F$ be the holomorphic function with $\Re F=4 C H_0$ on $\D$.
Then $G=Be^{F}$ is also holomorphic and non-vanishing on $\Omega$. For 
$z\in\partial \mathcal{D}_n^{2H_0}$, from the preceding case we know that
\[
 |G(z)|=|B(z)|e^{4 C H_0(z)}=|B_n(z)|\rho(z,\Lambda_n)e^{4C H_0(z)}\ge e^{-2CH_0(z)-H_0(\lambda_n)
  +4CH_0(z)}\ge 1\ .
\]
For $z \in \partial \D$ we have $|G(z)|= e^{4C H_0 (z)} \geq 1$.
Hence throughout $\Omega$
we have $|G|\ge 1$, that is,  $|B(z)|\ge |e^{-F(z)}|=e^{-4 C H_0(z)}$ for $z \in \Omega$.

\bigskip

(b)$\Longrightarrow$(c). We can assume that the function $H$ in (b) satisfies $\inf \{H(z) : z \in \D \} \geq \ln 3 $. Separate into two cases.

i) If $\rho(z,\Lambda) \geq e^{-10 H(z)}$ then, by hypothesis, $
|B(z)|+(1-|z|^2)|B'(z)|\geq |B(z)| \geq e^{-11 H(z)}.
$ 

ii) If $\rho(z,\Lambda) \leq e^{-10 H(z)}$ there exists a unique $\lambda_n$ such that such that $\rho(z,\Lambda)=\rho(z,\lambda_n)$. Then by hypothesis
\[
 \left|(1-\bar\lambda_n z)\frac{B(z)}{z-\lambda_n}\right|\geq e^{-H(z)} , \quad z\neq \lambda_n,
\]
and taking the limit as $z\to\lambda_n$, we deduce that $ (1-|\lambda_n|^2) |B'(\lambda_n)|\geq e^{-H(\lambda_n)}\ $. Finally, by Lemma~\ref{bounded-f}(c) and by Harnack's inequality \eqref{Harnack1} 
\[
 (1-|z|^2) |B'(z)|\geq e^{-H(\lambda_n)}-6\rho(z,\lambda_n)\geq e^{-2H(z)}-
  e^{-8 H(z)}\geq \frac{1}{2}e^{-2 H(z)}\geq e^{-3 H(z)},
\]
and therefore
\[
 |B(z)|+(1-|z|^2)|B'(z)|\geq e^{-11 H(z)}, \qquad z\in\D .
\]

(c)$\Longrightarrow$(a). This implication is immediate taking $z=\lambda_n$.

(a)$\Longrightarrow$(d). Let $H \in \Har_+(\D)$ such that $|B_n (\lambda_n) | \geq e^{- H(\lambda_n)}$, $n \in \N$, that is, 
 \bea \label{harmmaj}
  \sum_{k: k\neq n} \log\frac 1{\rho(\lambda_n,\lambda_k)}\leq  H(\lambda_n),\qquad n\in\N \ .
 \eea
Again the disks $\mathcal{D}_n^H$ are disjoint, and so will be the smaller
disks $\mathcal{D}_n^{4H}$.
By definition
\[
 \omega(\lambda_n, \partial\D, \Omega_n^H)=1- \sum\limits_{\stackrel{k: k\neq n}{\rho(\lambda_n,\lambda_k) \leq 1/2}} \omega(\lambda_n, \partial \mathcal D_{k}^H , \Omega_n^H)\ .
\]
Since
 \[
  \omega(z,\partial \mathcal D_{k}^{4H}, \D\setminus \mathcal D_{k}^{4H})=\frac{\log (1/\rho(z, \lambda_k))}{4H(\lambda_k)},
 \]
estimate \eqref{harmmaj} is equivalent to
\begin{equation}\label{int-har}
\sup_{n\in\N}  \sum_{k:k\neq n} \omega(\lambda_n,\partial \mathcal D_{k}^{4H}, \D\setminus \mathcal D_{k}^{4H})\, \frac{4H(\lambda_k)}{H(\lambda_n)}\leq 1 \ .
\end{equation}
If $\rho(\lambda_n,\lambda_k)\leq 1/2$ Harnack's inequalities \eqref{Harnack} imply that $1/3\leq H(\lambda_n)/H(\lambda_k)\leq 3$. Thus, by \eqref{int-har},
\begin{align*}
 \sum\limits_{\stackrel{k: k\neq n}{\rho(\lambda_n,\lambda_k) \leq 1/2}} \omega(\lambda_n, \partial \mathcal D_{k}^{4H} , \Omega_n^{4H})
 &\leq \sum\limits_{\stackrel{k: k\neq n}{\rho(\lambda_n,\lambda_k) \leq 1/2}} \omega(\lambda_n, \partial \mathcal D_{k}^{4H} ,\D\setminus \mathcal D_{k}^{4H})\\
 &\leq 3 \sum\limits_{\stackrel{k: k\neq n}{\rho(\lambda_n,\lambda_k) \leq 1/2}} \omega(\lambda_n, \partial \mathcal D_{k}^{4H} ,\D\setminus \mathcal D_{k}^{4H})\, \frac{H(\lambda_k)}{H(\lambda_n)}\leq \frac{3}{4}\ , 
\end{align*}
and therefore
\[
 \omega(\lambda_n, \partial\D, \Omega_n^{4H})\geq \frac{1}{4} \ .
\]

Observe that by replacing $4H$ by $NH$ in the above reasoning it is possible to get
$\omega(\lambda_n, \partial\D, \Omega_n^{NH})\geq 1-3/N$.
\bigskip

(d)$\Longrightarrow$ (a). For simplicity we drop the superscript $H$ in the notations $\mathcal D_n^H$ and  $\Omega_n^H$, and let $\delta_n=e^{-H(\lambda_n)}$. Let $\varepsilon=\inf\limits_{n\in\N} \omega(\lambda_n, \partial\D, \Omega_n)>0$ and consider the bigger domains
\[
 \tilde\Omega_n=\D\setminus \bigcup_{\stackrel{k: k\neq n}{\rho(\lambda_n,\lambda_k)\leq 1/4}} \mathcal D_{k} \ .
\]
Notice that then
$ \omega(\lambda_n, \partial\D, \tilde\Omega_n) \geq \omega(\lambda_n, \partial\D, \Omega_n) \geq \varepsilon\ .
$
Given $N\ge 1$, to be determined later on, let $\Delta_n = D(\lambda_n, \delta_n^N)
\subset D(\lambda_n,\delta_n)$ and 
\[
 \mathcal V_n=\D\setminus \bigcup_{\stackrel{k: k\neq n}{\rho(\lambda_n,\lambda_k)\leq 1/4}} \Delta_k\ .
\]
Notice that $\Omega_n \subset \tilde\Omega_n\subset \mathcal V_n$. Define the harmonic functions
\[
 U_n(z)=\omega(z,\partial \D, \Omega_n)\quad\textrm{and}\quad u_n(z)=\omega(z,\partial \D, \mathcal V_n)\ .
\]
Then $u_n(z)\geq U_n(z)$ for $\in\Omega_n$. In particular $u_n(\lambda_n)\geq \varepsilon>0, n\in\mathbb N.$
We apply Green's formula to the functions $\Phi(z)=\log(1/\rho(z,\lambda_n))$
and $u_n$ on the domain $\mathcal V_n$:
\begin{align*}
 u_n(\lambda_n)&=-\frac{1}{2\pi}\int_{\mathcal V_n}u_n\Delta\Phi dm\\
 &=\frac 1{2\pi}\int_0^{2\pi} P(\lambda_n, e^{i\theta}) d\theta-\sum\limits_{\stackrel{k: k\neq n}{\rho(\lambda_n,\lambda_k)\leq 1/4}} \frac 1{2\pi}\int_{\partial \Delta_k} \log\bigl(\frac 1{\rho(\lambda_n,\zeta)}\bigr)\frac{\partial u_n}{\partial n}(\zeta)\, d\sigma(\zeta),
\end{align*}
where $\partial/\partial n$ indicates the outer normal derivative. 
Using the hypothesis and the fact that for $\zeta\in\partial \Delta_k$ one has
\[
  \log\bigl(\frac 1{\rho(\lambda_n,\zeta)}\bigr)\asymp  \log\bigl(\frac 1{\rho(\lambda_n,\lambda_k)}\bigr),
\]
we deduce that
\beqa 
\lefteqn{\sum\limits_{\stackrel{k: k\neq n}{\rho(\lambda_n,\lambda_k)\leq 1/4}}  \log\bigl(\frac 1{\rho(\lambda_n,\lambda_k)}\bigr) \frac 1{2\pi}\int_{\partial \Delta_k} \frac{\partial u_n}{\partial n}(\zeta)\, d\sigma(\zeta)} \\
&&\qquad\lesssim
\sum\limits_{\stackrel{k: k\neq n}{\rho(\lambda_n,\lambda_k)\leq 1/4}} \frac 1{2\pi}\int_{\partial \Delta_k} \log\bigl(\frac 1{\rho(\lambda_n,\zeta)}\bigr)\frac{\partial u_n}{\partial n}(\zeta)\, d\sigma(\zeta)\leq 1-\varepsilon\ .
\eeqa 
Taking into account \eqref{interpol-local} we will be done as soon as we prove that $\frac{\partial u_n}{\partial n}(\zeta) \geq 0$, $\zeta \in \partial \Delta_k$ and 
\begin{equation}\label{green}
 \int_{\partial \Delta_k} \frac{\partial u_n}{\partial n}(\zeta)\, d\sigma(\zeta)\gtrsim \frac 1{H(\lambda_n)},\quad k\neq n\ .
\end{equation}
Define for $k\neq n$,
\begin{align*}
 u_{n,k}(z)&=\omega(z,\partial\D, \mathcal V_n\cup \Delta_k), \\
 v_k(z)&=\omega(z,\partial\Delta_k, \D\setminus \Delta_k)=\frac{\log(1/\rho(z,\lambda_k))}{\log (1/\delta_n^N)}\ 
\end{align*}
and notice that, again by the maximum principle,
\begin{equation}\label{u-unk}
 u_n\geq u_{n,k}-v_k\quad\textrm{on}\ \mathcal V_n\ . 
\end{equation}
For $\lambda_k,\lambda_j$ such that $\rho(\lambda_n,\lambda_k),\rho(\lambda_k,\lambda_j)\leq 1/4$ we have $\rho(\lambda_n,\lambda_j)\leq 1/2$ and therefore 
\[
 u_{n,k}(\lambda_k)\geq \omega\Big(\lambda_k,\partial\D, \D\setminus \bigcup_{\stackrel{j\neq k,n}{\rho(\lambda_j,\lambda_k)\leq 1/2}} \Delta_j\Big)\geq \omega(\lambda_k,\partial\D, \Omega_k)\geq \varepsilon>0\ .
\]
By Harnack's inequalities there exists $\varepsilon'=\varepsilon'(\varepsilon)>0$ such that
$
 u_{n,k}(z)\geq \varepsilon'>0$ for $ z \in  \partial \mathcal D_k
$.  
Also, for $z\in \partial\mathcal D_k$,
\[
 v_k(z)=\frac{\log(1/\delta_k)}{\log(1/\delta_k^N)}=\frac 1N\ 
\]
and inequality \eqref{u-unk} yields 
\[
 u_n(z)\geq u_{n,k}(z)-v_k(z)\geq \varepsilon'-\frac 1N, \quad z\in \partial \mathcal{D}_k.
\]
Choose $N$ so that $1/N<\varepsilon'/2$. Then
$u_n(z)\geq \varepsilon'/2$ for $z\in\partial\mathcal D_k$, $k\neq n$,
and by the maximum principle
\[
 u_n(z)\geq \frac{\varepsilon'}2 \omega_k(z),\quad z\in\mathcal D_k\setminus\Delta_k\ ,
\]
where
\[
 \omega_k(z)=\omega(z, \partial\mathcal D_k, \mathcal D_k\setminus\Delta_k)
=\frac{\log(\rho(z,\lambda_k)/\delta_k^N)}{\log(1/\delta_k^{N-1})}\ .
\]

Since $\log (1/\delta_k)=H(\lambda_k)$, this inequality implies that for $\zeta\in\partial\Delta_k$
\[
 \frac{\partial u_n}{\partial n}(\zeta)\geq \frac{\varepsilon'}2 \frac{\partial \omega_k}{\partial n}(\zeta)\gtrsim\frac 1{H(\lambda_k)}
 \frac{\partial}{\partial n}\log\rho(\zeta,\lambda_k)
\]
and therefore
\[
 \int_{\partial \Delta_k}  \frac{\partial u_n}{\partial n}(\zeta)\,  d\sigma(\zeta)\gtrsim \frac 1{H(\lambda_k)}\int_{\partial \Delta_k}
 \frac{\partial}{\partial n}\log\rho(z,\lambda_k)\,  d\sigma(\zeta)\ .
\]

Finally, we use Green's formula with $u\equiv 1$, $v(\zeta)=\log\rho(\zeta,\lambda_k)$  and the domain $\D\setminus\Delta_k$:
\[
\int_{\partial \Delta_k} \frac{\partial}{\partial n}\log\rho(\zeta,\lambda_k) \,  d\sigma(\zeta)=
\int_{\partial\D}\frac{\partial}{\partial n}\log\rho(\zeta,\lambda_k) \,  d\sigma(\zeta)=
\int_{\partial\D} P(\lambda_k,\zeta)\,  d\sigma(\zeta)=1\ .
\]
\end{proof}

We end this section with two easy consequences which will be used later. The first one says that Nevanlinna interpolating sequences are stable under convenient pseudohyperbolic perturbations, and will be deduced from Theorem~\ref{IntThm}(d). 

\begin{corollary}\label{stability}
 Let $\Lambda = \{\lambda_n \}$ be a Nevanlinna interpolating sequence and let $H\in\Har_+(\D)$,  satisfying Theorem~\ref{IntThm}(a).  If $\Lambda'=\{\lambda_n'\}_n\subset\mathbb D$ satisfies 
 \[
  \rho(\lambda_n, \lambda_n')\leq \frac 14 e^{-H(\lambda_n)}\ ,\quad n\in\mathbb N,
 \]
 then $\Lambda'$ is also a Nevanlinna interpolating sequence.
\end{corollary}

\begin{proof}
 We shall use the characterization of Nevanlinna interpolating sequences given in Theorem~\ref{IntThm}(d). Consider the domains
 \[
  \Omega_n=\mathbb D\setminus \bigcup_{\stackrel{k: k\neq n}{\rho(\lambda_k,\lambda_n)\leq 1/2}} D(\lambda_k, e^{-H(\lambda_k)}) \quad ,\quad
  \Omega_n'=\mathbb D\setminus \bigcup_{\stackrel{k: k\neq n}{\rho(\lambda_k^\prime,\lambda_n^\prime)\leq 1/4}} D(\lambda_k^\prime, e^{-2H(\lambda_k^\prime)}) 
 \]
 Then $\Omega_n\subset\Omega_n^\prime$, and by Harnack's inequality there exists $c>0$ such that
 \[
  \omega(\lambda_n^\prime, \partial\mathbb D, \Omega_n^\prime)\geq c\,\omega(\lambda_n, \partial\mathbb D, \Omega_n^\prime)
  \geq c\, \omega(\lambda_n, \partial\mathbb D, \Omega_n).
 \]
 The result follows then from the hypothesis.
\end{proof}

\begin{corollary}\label{lowerbound}
Let $\Lambda$ be a Nevanlinna interpolating sequence and let $H \in  \Har_+(\D)$ be such that  $\inf\{H(z) : z \in \D \} \geq \ln 3$ and 
$|B(z)|\ge e^{-H(z)}\rho(z,\Lambda)$, $z \in \D$. 
Then for every $H_1\in\Har_+(\D)$ with  $\inf\{H_1 (z) : z \in \D \} \geq \ln 3$, we have
\beqa
 \  |B(z)|\ge e^{-(2H(z)+2H_1(z))}\ \text{ whenever }\ z\notin\cup_n \mathcal{D}_n^{H_1}.
\eeqa
\end{corollary}

\begin{proof}[Proof of Corollary \ref{lowerbound}]
Suppose first $z\notin \cup_n D(\lambda_n,1/2)$. Then  $\rho(z,\Lambda)\ge 1/2$ and
\[
 |B(z)|\ge e^{-H(z)}\rho(z,\Lambda) \ge \frac{1}{2}e^{-H(z)}
 \ge e^{-2H(z)}.
\]
Next, if $z\in \cup_n D(\lambda_n,1/2)$ 
picking the closest point $\lambda_0 \in \Lambda$ with 
$\rho (z,\Lambda) =\rho(z,\lambda_0)\ge e^{-H_1(\lambda_0)}$, 
Harnack's inequality \eqref{Harnack1} gives 
\[
 |B(z)|\ge e^{-H(z)}\rho(z,\Lambda)=e^{-H(z)}\rho(z,\lambda_0)
 \ge e^{-H(z)-H_1(\lambda_0)}\ge e^{-H(z)-2H_1(z)}.
\]
\end{proof}

\section{Proof of Theorem \ref{MainThm}}\label{S3}

Notice first that we can assume throughout the proof that the functions $f_i$ are Blaschke products. For conditions (a), (b) and (d) this is easily seen by considering the Nevanlinna factorization $f_i = B_i  e^{g_i}$, where $B_i$ is the Blaschke product with the zeros of $f_i$ and $g_i$ is such that  $\Re(g_i)=H_i^+-H_i^-$, for some $H_i^+, H_i^-\in\Har_+(\D)$. Then, since $e^{g_i}$, $i=1,\ldots, m$, are invertible functions in $N$, we have $I(f_1,\dots, f_m)=I(B_1,\dots,B_m)$ and $J(f_1,\dots, f_m)=J(B_1,\dots,B_m)$.
As for condition (c), let us now see that there exists $H\in\Har_+(\D)$ such that
\bea \label{TolBlaschke}
 \sum_{i=1}^m( |B_i(z)|+(1-|z|^2) |B_i^\prime (z)|)\geq e^{-H(z)}\ ,\quad z\in\D
\eea
if and only if (c) holds with a suitable, possibly different, $H\in \Har_+(\D)$.

Let us first suppose that \eqref{TolBlaschke} holds. Let $E_i=e^{g_i}$, $i=1,\dots,m$, and take $H_1\in\Har_+(\D)$ such that
\[
 \bigl|\log|E_i (z)|\bigr|=|\Re (g_i (z))|\leq H_1 (z),  \quad z \in \D ,  \qquad i=1,\dots,m\ .
\]
Recall from Lemma \ref{Nevfcts} that
\begin{equation}\label{Ei}
 (1-|z|^2) |E_i^\prime (z)|\leq e^{C H_1(z)}\ ,\quad z\in\D,
\end{equation}
where $C>0$ is an absolute constant. Fix $z\in\D$. We shall distinguish two cases.

(i) Assume first that $z\in\D$ is such that
\[
 \sum_{i=1}^m |B_i(z)|\geq \frac 14 e^{- H(z)-(1+C) H_1(z)}\ .
\]
Then
\[
 \sum_{i=1}^m |f_i(z)|= \sum_{i=1}^m |B_i(z)| |E_i(z)| \geq \frac 14 e^{- H(z)-(2+C) H_1(z)}\ 
\]
and (c) holds.

(ii) Assume now that
\[
 \sum_{i=1}^m |B_i(z)|\leq \frac 14 e^{- H(z)-(1+C) H_1(z)}\ ,
\]
which is in particular bounded by $\frac{1}{4} e^{-H(z)}$.
Then by \eqref{TolBlaschke} we have 
\[
 \sum_{i=1}^m (1-|z|^2) |B_i^\prime(z)|\geq \frac 34 e^{- H(z)}\ .
\]
Therefore
\[
 \sum_{i=1}^m (1-|z|^2) |B_i^\prime(z)| |E_i(z)|\geq e^{-H_1(z)} \frac 34 e^{- H(z)}\ ,
\]
and by \eqref{Ei}
\[
 \sum_{i=1}^m (1-|z|^2) |B_i(z)| |E_i^\prime(z)|\leq  \frac 14 e^{- H(z)- H_1(z)}\ 
\]
Thus
\begin{align*}
  \sum_{i=1}^m (1-|z|^2) |f_i^\prime(z)|&\geq \sum_{i=1}^m (1-|z|^2) |B_i^\prime(z)| |E_i(z)|-\sum_{i=1}^m (1-|z|^2) |B_i(z)| |E_i^\prime(z)|\\
  &\geq \frac 34 e^{- H(z)- H_1(z)}-\frac 14 e^{- H(z)- H_1(z)}=\frac 12 e^{- (H(z)+ H_1(z))}\ 
\end{align*}
and so (c) holds.

The converse is based on exactly the same argument. Observe that we can write
$B_i=f_i/E_i=f_i\mathcal{E}_i$ where $\mathcal{E}_i$ is an invertible function in $N$ for which
we get similar estimates as for $E_i$.
Now, replacing in the arguments above $B_i$ by $f_i$ and $E_i$ by $\mathcal{E}_i$,
we will reach \eqref{TolBlaschke} when starting from (c).

\bigskip

Before giving the proof of Theorem~\ref{MainThm} we shall see that (a) implies 
$$I(f_1,\ldots, f_m)=
J(f_1,\ldots, J_m). $$
We only have to show the reverse inclusion.
For this, let $g\in J(f_1,\dots,f_m)$ and let $H\in\Har_+(\D)$ be such that
\[
 |g(z)|\leq e^{H(z)} \sum_{i=1}^m |f_i(z)|,  \quad  z \in \D .
\]
Let $B$ be a Nevanlinna interpolating Blaschke product in $I(f_1,\dots,f_m)$ and denote by $\Lambda=\{\lambda_n\}_n$ its zero set. 
Since for any $i=1, \ldots , m$, we have  
\[
 \frac{|g(\lambda_n)\overline{f_i(\lambda_n)}|}{\sum_{i=1}^m|f_i(\lambda_n)|^2 }\leq 
 \frac{e^{H(\lambda_n)} (\sum_{i=1}^m|f_i(\lambda_n)|)^2}{\sum_{i=1}^m|f_i(\lambda_n)|^2} \leq m e^{H(\lambda_n)},\qquad n\in \N ,
\]
using the description of the trace space $N|\Lambda$ in \eqref{TrSpace} we see that 
there exist $h_i\in N$ such that
\[
 h_i(\lambda_n)=\frac{g(\lambda_n)\overline{f_i(\lambda_n)}}
 {\sum_{i=1}^m|f_i(\lambda_n)|^2 },\qquad n\in \N .
\]
Consequently, the function $\sum_{i=1}^m f_i h_i - g$ vanishes on $\Lambda$, and therefore there exists $G\in N$ such that
\[
 \sum_{i=1}^m f_i h_i -g=BG\ .
\]
Since $BG\in I(f_1,\dots,f_m)$, this shows that $g\in I(f_1,\dots,f_m)$ as well.

\bigskip

Let us now move to the proof of Theorem~\ref{MainThm}.

(a) $\Lra$ (b) is obvious because $I(f_1, \ldots , f_m ) \subset J(f_1 , \ldots , f_m)$.

(b) $\Lra$ (c).  
Assume that $B\in J(f_1,\ldots,f_m)$ is a Nevanlinna interpolating Blaschke product and let $\Lambda=\{\lambda_n\}_n$ denote its zero set. By definition and by Theorem~\ref{IntThm}(b) there exist $H, H_1\in \Har_+(\D)$ such that 
\bea\label{nou}
 \rho(z,\Lambda)e^{-H_1(z)}\leq |B(z)|\le e^{H(z)} \sum_{i=1}^m |f_i(z)|,\quad z\in\D .
\eea
Recall from Lemma \ref{Nevfcts} that
there exists $H_2 \in \Har_+(\D)$ such that
\[
 |f_i(z)|+(1-|z|)|f_i'(z)|+(1-|z|)^2|f_i''(z)|\le e^{H_2 (z)} ,  \quad z\in\D , \, i=1, \ldots , m .
\]
Now let $H_3 \in \Har_+(\D) $, $H_3 \geq H + H_1 + H_2 + \ln 3$ to be chosen later. Observe that the disks $\mathcal{D}_n=\mathcal{D}_n^{H_3}
=D(\lambda_n,e^{-H_3(\lambda_n)})$ are disjoint. Observe also that \eqref{Harnack1} holds. 
By \eqref{nou} and  Corollary \ref{lowerbound}, we have 
\[
\sum_{i=1}^m |f_i(z)|\ge e^{-2(H(z)+H_1(z)+H_3(z))}\ ,\quad z\notin\bigcup_{n\ge 1} \mathcal{D}_n   .
\]
So, it only remains to discuss the estimate on $\mathcal{D}_n$. We will prove that
\bea\label{minor}
 \sum_{i=1}^m |f_i(z)|+(1-|z|)|f_i'(z)|\ge e^{-6H_3(z)},\quad z\in\mathcal{D}_{n}.
\eea
We argue by contradiction. Suppose there is a $z\in \mathcal{D}_n$ where this estimate does not hold. Let $u$ be the closest
point of $\partial\mathcal{D}_{n}=\partial\mathcal{D}_{n}^{H_3}$ to $z$, that is, $u \in \partial\mathcal{D}_{n}$ and $\rho(z,u) = \rho (z, \partial\mathcal{D}_{n})$.  
Then using a Taylor expansion at
$z$, as Tolokonnikov did in the $H^{\infty}$-case, for every $i=1, \ldots , m$, one has
\beqa
 |f_i(u)|&=&\bigl|f_i(z)+f_i'(z)(u-z)+\int_z^u(u-t)f_i''(t)dt\bigr|\\
 &\lesssim& |f_i(z)| + (1-|z|)|f_i'(z)| \rho(z,u)+(1-|z|)^2\sup_{v\in [z,u]}|f_i''(v)|\rho(z,u)^2\\
 &\lesssim& e^{-6H_3(z)}+e^{-6H_3(z)-H_3(\lambda_n)}+e^{ H_2 (v)-2H_3(\lambda_n)},
\eeqa
where $v$ is a suitable point in $\mathcal{D}_n$. Since $\rho(u,\Lambda)=e^{-H_3(\lambda_n)}$, using \eqref{nou} we deduce 
\[
 e^{-(H(u) +H_1(u))} e^{- H_3 (\lambda_n)} \le  \sum_{i=1}^m |f_i(u)|
 \lesssim m\left(2e^{-6H_3(z)}+e^{H_2 (v)-2H_3(\lambda_n)}\right).  
\]
Harnack's inequality \eqref{Harnack1} gives $H_3 (z) \geq H_3 (\lambda_n ) / 2$ and we deduce
\[
 e^{-(H(u) +H_1(u))} e^{- H_3 (\lambda_n)} \le  \sum_{i=1}^m |f_i(u)|
 \lesssim m\left(2e^{-3H_3(\lambda_n)}+e^{H_2 (v)-2H_3(\lambda_n)}\right).  
\]
Since the functions $H$, $H_1$ and $H_2$ are fixed and $H_3$ can be taken arbitrarily large, we obtain a contradiction. Hence \eqref{minor} holds and the statement (c) follows.

(c) $\Lra$ (a). First of all recall that in condition (c) we can assume that the functions $f_i$ are Blaschke products. 
We can also assume that the positive harmonic function $H$ appearing in condition (c) satisfies $\inf \{H(z) : z \in \D \} > \ln(3 m) $. Then Harnack's inequality \eqref{Harnack} gives that for any $h \in \Har_+ (\D)$ one has
\bea\label{Harnack2}
 \frac45 \leq \frac{h(z)}{h(w)} \leq \frac54 \quad \text{ if } \, \rho(w,z) < e^{-2H(z)} .
\eea
Now take $C>1$ big enough to be determined later on, and let 
\[
 E=\bigl\{z\in\D : \sum_{i=1}^m |f_i(z)|\leq e^{-CH(z)}\bigr\}=\cup_n E_n,
\]
where $E_n$ are the 
connected components of $E$. For every $n\in\mathbb N$ choose $\lambda_n\in E_n$, if
any, such that
\[
\sum_{i=1}^m |f_i(\lambda_n)|\le e^{-2CH(\lambda_n)},
\]
and let $\Lambda=\{\lambda_n\}_n$ (we discard those $E_n$ for which
such a $\lambda_n$ does not exist and keep the indexation with $\N$).
Observe that the sum above is trivially bounded by $e^{-2H(\lambda_n)}$.

\emph{Claim 1}. Assume $C \geq 24$. Then for every $\lambda_n\in\Lambda$, one has 
\[
 D(\lambda_n, e^{-2C H(\lambda_n)})\subset E_n\subset D(\lambda_n, e^{-6 H(\lambda_n)})\ .
\]

The first inclusion is an immediate consequence of Lemma~\ref{bounded-f}(a) and Harnack's inequality \eqref{Harnack2}, 
\beqa
 |f_i(z)|&\leq& |f_i(\lambda_n)|+2\rho(z,\lambda_n)\leq e^{-2C H(\lambda_n)}+2 e^{-2C H(\lambda_n)}
 =3e^{-2CH(\lambda_n)}\\
 &\le& 3e^{-(8/5)CH(z)}\le e^{-CH(z)}.
\eeqa
In order to see the second inclusion notice that, by hypothesis, on the set $E$, and so on $E_n$, 
the following estimate holds
\[
 (1-|z|^2) \sum_{i=1}^m |f'_i(z)|\geq e^{-2 H(z)},
\]
and in particular there exists $i$ such that $(1-|\lambda_n|^2)  |f'_i(\lambda_n)|
\geq e^{-2 H(\lambda_n)}/m\ge e^{-3H(\lambda_n)}=\delta$.
Thus by Lemma~\ref{bounded-f}(b), for every $z$ with $\rho(z,\lambda_n)=e^{-6H(\lambda_n)}
=\delta^2$ 
we have $|f_i(z)|\geq e^{-12H(\lambda_n)}$. By Harnack's inequality \eqref{Harnack1} we get
\bea\label{nbhdestim}
 |f_i(z)|\geq e^{-24 H(z)}\quad\textrm{if  }\rho(z,\lambda_n)=e^{-6H(\lambda_n)}.
\eea
Thus, taking $C\ge 24$, we get the desired inclusion.

Observe also that $\partial D(\lambda_n, e^{-6 H(\lambda_n)})\cap E=\emptyset$ and in particular
\[
 \rho(\lambda_n, \lambda_k)\geq\max(e^{-6H(\lambda_n)}, e^{-6 H(\lambda_k)}), \qquad k\neq n.
\]

\begin{lemma}
 The sequence $\Lambda$ constructed above is interpolating for $N$.
\end{lemma}

\begin{proof}
 We shall use the characterization given in Theorem~\ref{IntThm}(d). Consider the disks $\mathcal D_n^C=D(\lambda_n, e^{-2C H(\lambda_n)})$ and the domains
 \[
  \Omega_n^C=\D\setminus\bigcup\limits_{\stackrel{k:k\neq n}{\rho(\lambda_n,\lambda_k)\leq 1/2}} \mathcal D_k^C\ .
 \]
Since $\mathcal{D}_n^C\subset E_n$, Harnack's inequality \eqref{Harnack2} and the fact that $\|f_i \|_{\infty} \leq 1$ give that,  for every $i=1, \ldots , m$ 
we have
 \begin{align*}
  \log|f_i(\zeta)|&\leq -C H(\zeta)\leq- \frac{C}{2} H(\lambda_k)\quad \textrm{if}\ \zeta\in\partial\mathcal D_k^C,\\
  \log|f_i(\zeta)|&\leq 0 \qquad \textrm{if}\ \zeta\in\partial\D .
 \end{align*}
Hence, by the maximum principle
\[
 \log|f_i(z)|\leq-\frac{C}{2} \sum_{\stackrel{k: k\neq n}{\rho(\lambda_k,\lambda_n)\leq 1/2}} H(\lambda_k)\omega(z,\partial\mathcal D_k^C,\Omega_n^C)\ ,\qquad z\in\Omega_n^C.
\]
Notice that, by the separation above, the disk $D(\lambda_n, e^{-6 H(\lambda_n)})$ is contained in $\Omega_n^C$. Then, as established in \eqref{nbhdestim} there is $i$ such that
\[
 |f_i(\zeta)|\geq e^{-24 H(\zeta)}\qquad\textrm{if $\zeta\in \partial D(\lambda_n, e^{-6 H(\lambda_n)})\subset \Omega_n^C$},
\]
whence
\[
 \frac{C}{2} \sum_{\stackrel{k: k\neq n}{\rho(\lambda_k,\lambda_n)\leq 1/2}} H(\lambda_k)\omega(\zeta,\partial\mathcal D_k^H, \Omega_n^H)\leq24 H(\zeta), \quad \zeta \in \partial \mathcal D (\lambda_n , e^{-CH(\lambda_n)})
\]
By Harnack's inequality applied to both $H$ and $\omega(\cdot, \partial\mathcal D_k^C,\Omega_n^C)$, we deduce that
\[
 \sum_{\stackrel{k\neq n}{\rho(\lambda_k,\lambda_n)\leq 1/2}} \omega(\lambda_n, \partial\mathcal D_k^H,\Omega_n^H)\leq\frac{192}C\ .
\]
Choosing $C$ big enough we finally have
\[
 \omega(\lambda_n, \partial\D, \Omega_n^H)=1- \sum\limits_{\stackrel{k: k\neq n}{\rho(\lambda_n,\lambda_k) \leq 1/2}} \omega(\lambda_n, \partial \mathcal D_{k}^H , \Omega_n^H)\geq \frac 12 .
\]
\end{proof}

Notice that, by Theorem~\ref{IntThm}(a) and the proof above, there exists $C_0>0$ such that
\begin{equation*}\label{Bn}
 \prod_{k: k \neq n} \rho (\lambda_n , \lambda_k) \geq e^{-C_0 H(\lambda_n)},\quad n\in\mathbb N\ .
\end{equation*}
Although our choice of $\{\lambda_n \}_n$ depends on $C$, the constant $C_0$ is uniform. We indicate to the reader that the Nevanlinna interpolating Blaschke product we are
heading for is not constructed with the zero-set $\Lambda$ but with a sequence close
to $\Lambda$. This, in view of Lemma \ref{stability}, will guarantee that the new sequence is still interpolating. In the sequel we will need to introduce a new constant $D\gg C\gg C_0$, where $C \geq 24$ is the constant
fixed in the preceding discussions. Given a harmonic function $G$, we denote by $\tilde G$ its harmonic conjugate. 

\emph{Claim 2.} For every $n\in\mathbb N$ there exists $i\in\{1,\dots,m\}$ such that  $g_i:=f_i-e^{-12D(H+i\tilde H)}$ has a unique zero $a_n^{(i)}$ in $\mathcal D_n^{6H}$.

By condition (c) we can assume that for some $i\in\{1,\dots,m\}$ (not necessarily unique) we have $(1-|\lambda_n|^2)|f_i^\prime(\lambda_n)|\geq e^{-3H(\lambda_n)}$. 
Since $|f_i (\lambda_n)|\le e^{-2CH(\lambda_n)}$ and $C \geq 6$,
applying again Lemma \ref{bounded-f}(b) we obtain
\bea\label{inferior}
 |f_i(z)|\geq e^{-12H(z)} \quad \textrm{for $z\in \partial \mathcal D_{n}^{6H}$}.
\eea
We use this and Rouch\'e's theorem to compare the number of zeros of $g_i$ and the function $h_i=f_i-f_i(\lambda_n)$ in $\mathcal D_{n}^{6H}$. Observe that $h_i$ vanishes at $\lambda_n $ and 
$(1-|\lambda_n|)|h_i '(\lambda_n)|\ge e^{-3H(\lambda_n)}$, so that with Lemma \ref{bounded-f}(b),
applied to any $\delta<e^{-6H(\lambda_n)}$, it can be shown that $h_i$ does not vanish at any
other point of $\mathcal{D}_n^{6H}$. Now, 
for $z\in \partial \mathcal D_{n}^{6H}$, Harnack's inequality \eqref{Harnack2}, $D \geq C \geq 24$ and \eqref{inferior} give
\begin{align*}
 |g_i(z) -\left(f_i(z)-f_i(\lambda_n)\right)|&=|f_i(\lambda_n)-e^{-12D\, (H+i\tilde H)(z)}|\leq  e^{-2CH(\lambda_n)}+e^{-12D H(z)}\\
 &\leq e^{-C H(z)}< |f_i(z)-f_i(\lambda_n)|\ ,
\end{align*}
as desired. This proves the Claim.  $\square$

The argument works for every $i$ with $(1-|\lambda_n|^2)|f_i^\prime(\lambda_n)|\geq e^{-3H(\lambda_n)}$, but we will pick $a_n^{i}$ for only one $i$. We will denote by $i(n)$ the index in $ \{1, \ldots , m\}$ satisfying Claim 2. The previous argument with Rouch\'e's theorem also allows to show that $\rho(a_n^{(i)},\lambda_n)\leq e^{-CH(\lambda_n)}$. Since $C\gg C_0$, we deduce from Lemma~\ref{stability} that the sequence $A_i:=\{a_n^{(i)}\}_n$ is also interpolating for $N$. 
Let $I_i$ denote the Nevanlinna interpolating Blaschke product with zero set $A_i$.

\emph{Claim 3.} Assume $C \geq 24$. Then $\sum_{j=1}^m |g_j (z)| / |I_j (z)| \geq e^{-4C H(z)}$ for any $z \in \D$.

To see this consider first $z\notin \cup_n \mathcal D_n^{6H}$, so that $\sum_{i=1}^m|f_i(z)|> e^{-CH(z)}$. Hence, there exists $f_i$ such that $|f_i(z)|\geq e^{-2CH(z)}$, and therefore
\bea\label{goverI}
\quad \sum_{j=1}^m\left|\frac{g_j(z)}{I_j(z)}\right|\geq |g_i(z)|\geq |f_i(z)|-e^{-12DH(z)}\geq e^{-2CH(z)}-e^{-12DH(z)}\geq e^{-3CH(z)}.
\eea

Consider now $z\in\mathcal D_n^{6H}$.
Notice first that for $\zeta \in\partial\mathcal D_n^{6H}$ and for $i=i(n)$ by \eqref{inferior}, we have
$|g_i(\zeta)|\ge e^{-3CH(\zeta)}$.
Applying the minimum modulus principle to $g_i/I_i$ we deduce that
\[
 \left |\frac{g_i(z)}{I_i (z)}\right|\geq \inf_{\zeta \in \partial \mathcal D_n^{6H} } |g_i (\zeta)| \geq   \inf_{\zeta \in \partial \mathcal D_n^{6H} } e^{-3C H(\zeta)}  \geq e^{-4C  H(z)}.
\]
This finishes the proof of the Claim. 
$\square$

Since $(g_1/I_1,\ldots,g_m/I_m)$ is unimodular, by the Corona Theorem for $N$ (see Introduction), there exist $h_i\in N$ such that 
\[
 \sum_{i=1}^m \frac{g_i}{I_i}\, h_i \equiv 1\ 
\]
and
\[
\sum_{i=1}^m |h_i (z)| \leq e^{M_0 H(z)} , \, z \in \D \, . 
\]
Here $M_0 = M_0 (C) >0$ is a constant which may depend on $C$ but not $D$, since the estimate in Claim 3 only depends on $C$. 
Since $g_i = f_i - e^{-12 D (H + i \tilde{H})}$, we have 
\begin{equation}\label{function}
 F:=\sum_{i=1}^m f_i (h_i\frac{\prod_{k=1}^m I_k}{I_i}) = \prod_{k=1}^m I_k +e^{-12D\, (H+i\tilde H)} \sum_{i=1}^m h_i\frac{\prod_{k=1}^m I_k}{I_i} \ .
\end{equation}
Since the function $F$ 
is obviously in $I(f_1,\ldots,f_m)$,
we will be done as soon as we show that the zero set of this function is an interpolating sequence for $N$. In order to consider the zeros of $F$ we will again distinguish two cases. 

Observe first that since $\rho(a_n^{(i)}, \lambda_n)\leq e^{-C H(\lambda_n)}$, choosing $C\gg C_0$, 
and observing that $A_i$ are Nevanlinna interpolating sequences, we will have
\bea\label{Rouche2}
 |\prod_{k=1}^m I_k(z)|\geq e^{-2C_0 H(z)}\quad\textrm{for $z\in\D\setminus\cup_n \mathcal D_n^{6H}$}.
\eea
Since 
\[
 |e^{-12D\, (H+i\tilde H)} \sum_{i=1}^m h_i\frac{\prod_{k=1}^m I_k}{I_i} |
 \le e^{(-12D+M_0)H},
\]
which, choosing $D$ large enough, can be assumed neglectible with respect to $e^{-C_0H}$, we see that $F$ cannot vanish outside
the disks $\mathcal D_n^{6H}$.

 To consider the disks $\mathcal D_n^{6H}$, we again use Rouch\'e's theorem to see that $F$ has exactly one zero in such a disk. 
Since $A_i$ is Nevanlinna interpolating we can then conclude by applying the stability result Lemma~\ref{stability}. To apply Rouch\'e's theorem we shall compare the function \eqref{function} with $\prod_{k=1}^m I_k$.
In view of \eqref{Rouche2}, for every $n\in\mathbb N$ and $z\in\partial \mathcal D_n^{6H}$
\begin{align*}
 \Bigl| \prod_{k=1}^m I_k(z) &+ 
 e^{-12D\, (H+i\tilde H)(z)} \sum_{i=1}^m h_i (z) \frac{\prod_{k=1}^m I_k(z)}{I_i(z)}-\prod_{k=1}^m I_k(z)\Bigr|\\
 & \leq e^{-12DH(z)} \sum_{i=1}^m |h_i(z)|\leq e^{-(12D-M_0) H(z)}\leq e^{-2C_0 H(z)}  < \left| \prod_{k=1}^m I_k(z)\right|, 
\end{align*}
as desired.\hfill\qedsymbol


\section{The case of two generators}\label{S3b}

In this section we shall assume $m=2$ and prove the equivalence between condition (d) and (a),(b), (c) in Theorem~\ref{MainThm}. We have already proved that (a) implies that $I(f_1 , \ldots , f_m) = J(f_1 , \ldots , f_m)$ for any $m \geq 2$. Hence we only need to prove the sufficiency of condition (d) when $m=2$. We start with an auxiliary result which allows to reduce the situation to the case where $B_1$ and $B_2$ have no
common zeros.

\begin{lemma}\label{red}
Let $\hat{B}$ be the Blaschke product formed with the common zeros of $f_1$ and $f_2$. Then
$I(f_1,f_2)=J(f_1,f_2)$ if and only if $I(f_1/\hat{B},f_2/\hat{B})=J(f_1/\hat{B},f_2/\hat{B})$.
\end{lemma}

\begin{proof}
If $f\in J(f_1/\hat{B},f_2/\hat{B})$, then $|f|\le e^H(|f_1/\hat{B}|+|f_2/\hat{B}|)$ for some $H \in \Har_+ (\D)$,
and so $f\hat{B}\in J(f_1,f_2)=I(f_1,f_2)$ giving $f\in I(f_1/\hat{B},f_2/\hat{B})$.
Conversely, if $f\in J(f_1,f_2)$, then $|f|\le e^H(|f_1|+|f_2|)$, for some $H \in \Har_+ (\D)$.
In particular  $\hat{B}$ divides $f$. 
Hence $f/\hat{B}\in J(f_1/\hat{B},f_2\hat{B})=I(f_1/\hat{B},f_2/\hat{B})$ 
giving $f\in I(f_1,f_2)$.
\end{proof}

In order to prove the sufficiency of condition (d) when $m=2$ we need some more auxiliary results.

\begin{lemma}\label{41}
Let $0<m_1\le m_2\le \cdots\le m_N\le 1$, with $N\ge 2$. Assume $\prod_{j=1}^Nm_j\le \eta<1$
and $\prod_{j=2}^N m_j\le \eta^{1/2}$. Then there exists an integer $k$ with $1\le k<N$ such that
$\prod_{j=1}^km_j\le \eta^{1/4}$ and $\prod_{j=k+1}^N m_j\le \eta^{1/2}$ ($\le \eta^{1/4}$).
\end{lemma}

\begin{proof}
Let $k$ be the smallest positive integer such that $ \prod_{j=1}^k m_j\le \eta^{1/4}$.
Observe that $k < N$, because otherwise $\prod_{j=1}^{N-1}m_j>\eta^{1/4}$ and it would follow that
$m_N<\eta^{3/4}$, and then 
\[
 \prod_{j=1}^{N-1}m_j<m_1\le m_N<\eta^{3/4}<\eta^{1/4},
\]
which is a contradiction. Hence $k<N$.
If $k=1$, the conclusion follows immediately from the assumption $\prod_{j=2}^Nm_j\le 
\eta^{1/2}$. 
Next, if $k>1$, we have $\prod_{j=1}^{k-1}m_j>\eta^{1/4}$ and $\prod_{j=1}^k m_j\le 
\eta^{1/4}$.
The first estimate gives $m_1>\eta^{1/4}$ and hence $m_j>\eta^{1/4}$ for any
$j=1,\ldots,N$. Then
\[
 \prod_{j=k+1}^Nm_j=\frac{\prod_{j=1}^Nm_j}{m_k\prod_{j=1}^{k-1}m_j}
 \le\frac{\eta}{\eta^{1/4}\eta^{1/4}}=\eta^{1/2}.
\]
\end{proof}

\begin{lemma}\label{ffgg}
Let $f_i,g_i\in N$, $i=1,2$, such that $f_1g_1$ and $f_2g_2$ have no common zeros. 
If $I(f_1g_1,f_2g_2)=J(f_1g_1,f_2g_2)$, then $I(f_1,f_2)=J(f_1,f_2)$.
\end{lemma}

\begin{proof}
We need to show that $J(f_1,f_2)\subset I(f_1,f_2)$. Let $f\in J(f_1,f_2)$, that is
$|f|\le e^H(|f_1|+|f_2|)$, for some $H\in \Har_+(\D)$. Then there exists another 
$H_1\in\Har_+(\D)$ such that $|fg_1g_2|\le e^{H_1}(|f_1g_1|+|f_2g_2|)$. By assumption, there exist $h_1,h_2
\in N$, such that $fg_1g_2=f_1g_1h_1+f_2g_2h_2$. Thus $f_1g_1h_1$ vanishes at the
zeros of $g_2$, and since $f_1g_1$ and $f_2g_2$ have no common zeros, so that $f_1g_1$ and $g_2$
have no common zeros, it is $h_1$ vanishing
at the zeros of $g_2$. We thus may write $h_1=g_2h_1^*$ for a suitable $h_1^*\in N$.
A similar argument leads to $h_2=g_1h_2^*$ for some $h_2^*\in N$. Thus
$f=f_1h_1^*+f_2h_2^*$.
\end{proof}

\begin{lemma}\label{Blaschke-no-zeros}
 Let $B$ be a Blaschke product with zero sequence $\Lambda$. Let $z\in\D$ be such that $\Lambda\cap D(z,\delta)=\emptyset$ and let $\rho_{\Delta}$ denote the
pseudohyperbolic distance in $\Delta=D(z,\delta)$. Then
\begin{itemize}
 \item [(a)] $|B(z)|^{\frac{1+\rho_{\Delta}(z,w)}{1-\rho_{\Delta}(z,w)}}
 \le |B(w)|
 \le |B(z)|^{\frac{1-\rho_{\Delta}(z,w)}{1+\rho_{\Delta}(z,w)}}, \quad w\in \Delta $,
 
 \bigskip
  \item [(b)] $ (1-|z|^2)|B'(z)|\le \dfrac{|B(z)|}{\delta}\log \dfrac 1{|B(z)|^{2}}$ .
\end{itemize}
\end{lemma}

 \begin{proof}
The estimates in (a) are just Harnack's inequalities rescaled to $\Delta$ and applied to the positive harmonic function $u=-\log|B|$. To prove (b) let $\Lambda = \{\lambda_n \}_n$. A direct computation shows that
\[
   B'(z)=\sum_{n=1}^\infty \frac{B(z)}{b_{\lambda_n}(z)}
 \frac{-\overline{\lambda}_n}{|\lambda_n|}\frac{1-|\lambda_n|^2}{(1-\bar\lambda_n z)^2}\ .
\]
Hence 
\[
   (1-|z|^2)|B'(z)|\leq \sum_{n=1}^\infty \frac{|B(z)|}{\delta} \frac{(1-|z|^2)(1-|\lambda_n|^2)}{|1-\bar\lambda_n z|^2}
\]
and we finish by using the estimate $\log(1/x)\geq 1-x$, $x>0$, since then
\[
 \sum_{n=1}^\infty \frac{(1-|z|^2)(1-|\lambda_n|^2)}{|1-\bar\lambda_n z|^2}= \sum_{n=1}^\infty 
 (1-\rho^2(\lambda_n,z))\leq 2 \sum_{n=1}^\infty \log\frac 1{\rho(\lambda_n,z)}
=\log\frac 1{|B(z)|^2}\ .
\]
 \end{proof}

\begin{lemma}\label{diff-div}
Let $\Lambda=\{\lambda_n\}_n$ be a sequence of distinct points in $\D$ which is the union
of two Nevanlinna interpolating sequences. Then the trace of $N$ on $\Lambda$ is
\[
 N|\Lambda=\left\{ \{w_n\}_n\ :\ \exists H\in \Har_+(\D)\, :\, \sup_{k:k\neq n}\frac{|w_k-w_n|}{\rho(\lambda_k,\lambda_n)}e^{-H(\lambda_n)-H(\lambda_k)}
 <\infty\right\} .
\]
\end{lemma}
It is also true in general that when $\Lambda$ is the union of $n$ Nevanlinna interpolating sequences then the trace coincides with the set of sequences such that the pseudohyperbolic divided differences of order $n-1$ have a positive harmonic majorant (see \cite{HMN}).

\begin{proof} $\boxed{\subseteq}$ Let $\{w_n\}_n\in N|\Lambda$ and let $f\in N$ with $f(\lambda_n)=w_n$, $n\in\mathbb N$. Let $H$ be a positive harmonic majorant of $\log |f|$. Given $\lambda_n,\lambda_k\in\Lambda$, $k\neq n$. Define
\[
 \Delta f(\lambda_n,\lambda_k)=\frac{f(\lambda_k)-f(\lambda_n)}{b_{\lambda_n}(\lambda_k)}.
\]

If $\rho(\lambda_n,\lambda_k)\geq 1/2$ we get
\[
 |\Delta f(\lambda_n,\lambda_k)|
=\left|\frac{f(\lambda_k)-f(\lambda_n)}{b_{\lambda_n}(\lambda_k)}\right|
 \leq \frac 1{\rho(\lambda_n,\lambda_k)}\bigl(e^{H(\lambda_k)}+e^{H(\lambda_n)}\bigr)
\leq  2 e^{H(\lambda_k)+H(\lambda_n)}\ .
\]

If $\rho(\lambda_n,\lambda_k)< 1/2$ apply the maximum principle to the holomorphic
function $z\lmto \Delta f(\lambda_n,z)$ and use Harnack's inequalities \eqref{Harnack1} to get
\begin{align*}
  |\Delta f(\lambda_n,\lambda_k)|&\leq \sup_{\zeta : \rho(\lambda_k,\zeta)=1/2} 
 |\Delta f(\lambda_n,\zeta)|\leq \sup_{\zeta : \rho(\lambda_k,\zeta)=1/2} 2 e^{H(\lambda_n)+H(\zeta)} \leq e^{3H(\lambda_n)+3H(\lambda_k)}\ .
\end{align*}

$\boxed{\supseteq}$
Let $\Lambda=\Lambda_1\cup\Lambda_2$, where $\Lambda_i=\{\lambda^{(i)}_n\}_n$ are Nevanlinna interpolating sequences, $i=1,2$,
and denote by $B_i$ the corresponding Blaschke products. We will also denote $w_k^{i}=w_n$ when $\lambda_k^{(i)}=\lambda_n$.
A usual technique to interpolate on finite unions of  interpolating sequences is to look for 
an interpolating function of the form $h_0+B_1h_1$, where $h_0$ interpolates on $\Lambda_1$
and $h_1$ interpolates suitable values controlled by the divided differences on $\Lambda_2$.
Since by assumption $\{w_k^{(1)}\}_k$ has a majorant $e^{H(\lambda_k^{(1)})}$, 
there exists $h_0\in N$ with $h_0(\lambda_k^{(1)})=w_k^{(1)}$, $k\in\mathbb N$.
If we want an interpolating function of the form $h=h_0 +B_1 h_1$, with $h_1\in N$, then,
 $h(\lambda_k^{(2)})=w_k^{(2)}$ reduces to 
\bea\label{divdiff}
 h_1(\lambda_k^{(2)})=\frac{w_k^{(2)}-h_0(\lambda_k^{(2)})}{B_1(\lambda_k^{(2)})},\qquad k\in\mathbb N.
\eea
Since $\Lambda_2\in\Int N$ we only need to see that the values on the right hand side have a suitable majorant.
Given $\lambda_k^{(2)}$ take $\lambda_k^{(1)}$ such that $\rho(\lambda_k^{(2)},\Lambda_1)=\rho(\lambda_k^{(2)},\lambda_k^{(1)})$. There is no restriction in assuming that $\rho(\lambda_k^{(2)},\lambda_k^{(1)})\leq 1/2$, since otherwise the estimate below is immediate. Since $\Lambda_2$ is a Nevanlinna interpolating sequence, 
by Theorem~\ref{IntThm}(b), there exists $H_1\in \Har_+(\D)$ such that 
\[
 |B_1(\lambda_k^{(2)})|\ge e^{-H_1(\lambda_k^{(2)})}\rho(\lambda_k^{(1)},\lambda_k^{(2)}), \qquad k\in\mathbb N,
\]
and therefore
\begin{align*}
 \left|\frac{w_k^{(2)}-h_0(\lambda_k^{(2)})}{B_1(\lambda_k^{(2)})}\right| &\leq
 \left|\frac{w_k^{(2)}-w_k^{(1)}}{B_1(\lambda_k^{(2)})}\right|
 +\left|\frac{h_0(\lambda_k^{(1)})-h_0(\lambda_k^{(2)})}{B_1(\lambda_k^{(2)})}\right|\\
 &\leq \left(\frac{|w_k^{(2)}-w_k^{(1)}|}{\rho(\lambda_k^{(1)},\lambda_k^{(2)})} +
 \frac{|h_0(\lambda_k^{(1)})-h_0(\lambda_k^{(2)})|}{\rho(\lambda_k^{(1)},\lambda_k^{(2)})}\right) e^{H_1(\lambda_k^{(2)})}\ .
\end{align*}
By hypothesis the first term between parentheses has a majorant of the form $e^{H(\lambda_k^{(1)})+H(\lambda_k^{(2)})}$. The second term can be assumed to satisfy the same estimate because of the first inclusion and the fact that $h_0\in N$. Thus, there exists $H_2 \in \Har_+ (\D)$ such that
\[
  \left|\frac{w_k^{(2)}-h_0(\lambda_k^{(2)})}{B_1(\lambda_k^{(2)})}\right|\leq 2e^{H_2 (\lambda_k^{(1)})+H_2 (\lambda_k^{(2)})}e^{H_1(\lambda_k^{(2)})} \, .
\]
By Harnack's inequality this is bounded by $2 e^{2 H_2(\lambda_k^{(2)})} e^{H_1(\lambda_k^{(2)})}$. Then \eqref{TrSpace} yields the existence of $h_1$ such that \eqref{divdiff} holds.
\end{proof}

\begin{lemma}\label{45}
Let $\Lambda=\{\lambda_n\}_n$ be a separated Blaschke sequence and let
$\delta:=\inf_{k\neq n}\rho(\lambda_k,\lambda_n)>0$.
Given $0<\eps_n<\delta/2$ consider the disks $\mathcal D_n=D(\lambda_n,\eps_n)$. 
Let $B_1$ and $B_2$ be two Blaschke products without common zeros, having each exactly two zeros  in each disk
$\mathcal D_n$. Assume $I(B_1,B_2)=J(B_1,B_2)$. Then there exists $H\in\Har_+(\D)$ such that
\[
 \eps_n>e^{-H(\lambda_n)},\quad n\in \N.
\]
\end{lemma}

\begin{proof}
The assumptions {\it a priori} allow $B_1$ and $B_2$ to have zeros outside
$\cup_{n}\mathcal{D}_n$. In order to get rid of these,
let $h_i$ be the Blaschke product vanishing on the zeros of $B_i$ which are not in $\cup_n
\mathcal{D}_n$. Setting $B_i^0=B_i/h_i$,  Lemma~\ref{ffgg} shows that $I(B_1^0,B_2^0)=J(B_1^0,B_2^0)$ (note that $B_1$ and $B_2$ are assumed to have no common zeros). Thus we can henceforth assume that the zeros of $B_i$, $i=1,2$, are contained in $\cup_n \mathcal D_n$. 

Let $c_n^i, d_n^i$ denote the zeros of $B_i$ in $\mathcal D_n$, $n\in\N$, $i=1,2$. Pick the largest of the mutual distances 
 $\rho(c_n^1, c_n^2),\ \rho(c_n^1, d_n^2),\ \rho(d_n^1, c_n^2),\ \rho(d_n^1, d_n^2)$, say $\rho(d_n^1, d_n^2)$. Then we have
 \bea\label{bigdist}
  2\eps_n\geq \rho(d_n^1, d_n^2)\geq\max\{\rho(c_n^1, c_n^2),\ \rho(c_n^1, d_n^2),\ \rho(d_n^1, c_n^2)\}.
 \eea
For $i=1,2$ let $D_i$ be the Blaschke product with zeros $\{d_n^i\}_n$ and let $C_i=B_i/D_i
=\prod b_{c_n^i}$. Since $\Lambda$ is separated and $B_i$ has exactly two zeros on each $\mathcal{D}_n$ we deduce from \cite[Corollary 1.9]{HMNT} that $C_i$ and $D_i$ are Nevanlinna interpolating Blaschke products.
Hence, taking into account \eqref{bigdist} there exists $H\in\Har_+(\D)$ such that the values
\[
 \left|\frac{C_1(d_n^2)}{D_1(d_n^2)}\right|
 =\frac{\rho(c_n^1,d_n^2)}{\rho(d_n^1,d_n^2)}  \left|\frac{(C_1/b_{c_n^1})(d_n^2)}{(D_1/
 b_{d_n^1})(d_n^2)}\right|
\ ,\qquad  \frac{C_2(d_n^1)}{D_2(d_n^1)}
 =\frac{\rho(c_n^2,d_n^1)}{\rho(d_n^2,d_n^1)}  \left|\frac{(C_2/b_{c_n^2})(d_n^1)}{(D_2/
 b_{d_n^2})(d_n^1)}\right|
\]
are bounded by $e^{H(d_n^2)}$ and $e^{H(d_n^1)}$,  respectively. Consequently, there exist $h_1,h_2\in N$ such that
\[
 h_1(d_n^2)=\frac{C_1(d_n^2)}{D_1(d_n^2)}\ ,\qquad  h_2(d_n^1)=\frac{C_2(d_n^1)}{D_2(d_n^1)}\ .
\]
Hence, there are $g_1,g_2\in N$ with
$C_1 = D_1 h_1 + D_2 g_1$ and $C_2 = D_2 h_2 + D_1 g_2$. 
Next we show that $C_1C_2\in J(B_1,B_2)$. Indeed, assume (without loss of generality) that $|C_2(z)|\leq |C_1(z)|$. Then
\begin{align*}
 |C_1(z) C_2(z)|\leq |(D_1h_1)(z)+(D_2g_1)(z)| |C_2(z)|\leq |h_1(z)||B_1(z)|+ |g_1(z)||B_2(z)|\ .
\end{align*}

Hence $C_1C_2\in J(B_1,B_2)=I(B_1,B_2)$ so that there exist $f_1,f_2\in N$ with
\[
 C_1 C_2 = B_1 f_1+B_2 f_2=C_1 D_1f_1 + C_2 D_2 f_2\ .
\]
Therefore, $f_1$ vanishes at the zeros of $C_2$ and $f_2$ vanishes at the zeros of $C_1$, and
there exist $f_1^*,f_2^*\in N$ with
$
 f_2=C_1 f_2^* $ and  $f_1=C_2 f_1^*$ .
Hence
\[
 C_1 C_2= C_1 D_1 C_2 f_1^*+ C_2 D_2 C_1 f_2^*
\]
and we deduce that
$
 1= D_1 f_1^*+ D_2 f_2^*\ .
$
Then there exists $H_1 \in\Har_+(\D)$ such that
 $ |D_1|+|D_2|\geq e^{-H_1}\ .
$
Consequently, and since $\rho(d_n^1, d_n^2)\leq \eps$ we can use Harnack's inequalities to deduce that
\[
 \eps_n\geq \rho(d_n^1, d_n^2)\geq |D_1(d_n^2)|\geq e^{-H(d_n^2)}\geq e^{-2H(\lambda_n)}\ .
\]
\end{proof}

Let us now move to the proof of (d)$\Lra$(c) in Theorem \ref{MainThm} in the case $m=2$. Recall that we can assume that $f_i=B_i$ are Blaschke products. Let $\Lambda_i$ be the zero set of $B_i$ and denote
\[
 k(z)=\sum_{i=1}^2 (|B_i(z)|+(1-|z|^2) |B_i^\prime(z)|)\, ,  \qquad z\in \D .
\]
In view of \cite[Proposition 4.1]{HMNT}, for any $\delta>0$ there exists $H_\delta\in\Har_+(\D)$ such that
\[
 |B_i(z)|\geq e^{-H_{\delta}(z)}\qquad\textrm{for $z$ with  $\ \rho(z,\Lambda_i)\ge \delta$, $i=1,2$}\ .
\]
Hence, to prove estimate (c) we can assume that $z$ belongs to a Whitney box $T(I)=
\{z=re^{i\theta}\in\D:e^{i\theta}\in I, |I|/2\le 1-r\le |I|\}$ such that $\rho(T(I), \Lambda_i)\leq 1/2$, $i=1,2$. Here $I$ indicates an arc in $\partial\D$. 
Let $\{T(I_j)\}_j$ be the collection of Whitney boxes  satisfying this condition and pick $\alpha_j\in \overline{T(I_j)}$ such that 
\[
 k(\alpha_j)=\min_{z\in  \overline{T(I_j)}} k(z)\ .
\]
To prove (c) we need to construct $H\in\Har_+(\D)$ such that
\begin{equation}\label{A}
 k(\alpha_j)\geq e^{-H(\alpha_j)} , \quad j\in\N\ ,
\end{equation}
since then, by Harnack's inequalities, the inequality propagates to the whole $\overline{T(I_j)}$, that is, if $z\in \overline{T(I_j)}$, we have 
$
 k(z)\geq k(\alpha_j) \geq e^{-H(\alpha_j)} \geq e^{-CH(z)} .
$

Splitting $\{\alpha_j\}_j$ into finitely many subsequences if necessary, one can assume that the pseudohyperbolic disks $\mathcal  D_j=D(\alpha_j, 1/2)$ are pairwise disjoint. 

For $i=1,2$ and $j\in\N$ let $B_i(j)$ be the subproduct of $B_i$ formed with the zeros of $B_i$ placed outside $\mathcal D_j$. Then (again using \cite[Proposition 4.1]{HMNT}, see also \eqref{interpol-local}) there exists $H_0\in\Har_+(\D)$ independent of $i$ and $j$ such that
\begin{equation}\label{B}
 |B_i(j)(\alpha_j)|\geq e^{-H_0(\alpha_j)}\ ,\quad j\in\N .
\end{equation}
We can also assume that each $\mathcal D_j$ contains at least two zeros of $B_1$ and two zeros of $B_2$. 
Indeed, suppose $\lambda$ is the only zero of $B_1$ in 
$\mathcal{D}_j$ (if there is none, then $B_1(j)=B_1$ and $k(\alpha_j)\ge |B_1(\alpha_j)|
\ge e^{-H(\alpha_j)}$ so that there is nothing to do). 
If $\rho(\alpha_j,\lambda)\ge e^{-H_1(\alpha_j)}$ for a suitable fixed $H_1$, then since
$|B_1 |=|B_1 (j)| |b_{\lambda}|$ we get \eqref{A}. 
If  $\rho(\alpha_j,\lambda)\le e^{-H_1(\alpha_j)}$, first  observe that
$(1-|\lambda|^2)|B_1'(\lambda)|=|B_1(j)(\lambda)|\ge e^{-2 H_0(\lambda)}$.  
Then, by Lemma \ref{Nevfcts}(b) we deduce that, for a sufficiently big $H_1$ (depending on 
$H_0$ only), $(1-|\alpha_j|)|B_1'(\alpha_j)||\ge e^{-3H_0(\alpha_j)}$,
which again yields \eqref{A}.  

We can also assume that 
\bea\label{k100}
 k(\alpha_j)\leq e^{-100 H_0(\alpha_j)},
\eea 
since otherwise \eqref{A} holds.

For $i=1,2$ and $j\in\N$ let $\lambda_j^{(i)}$ be a zero of $B_i$ such that $\rho(\alpha_j, \lambda_j^{(i)})=\rho(\alpha_j,\Lambda_i)$. Denote $B_{i,j}=B_i/b_{\lambda_j^{(i)}}$. 
We claim that there exists a universal constant $C>0$ such that 
\begin{equation}\label{C}
 |B_{i,j}(\alpha_j)|\leq C k(\alpha_j)^{1/2},\qquad j\in\N,\, i=1,2\ .
\end{equation}
To see this notice first that we have $|B_i(\alpha_j)|\leq k(\alpha_j)$. If $\rho(\alpha_j, \lambda_j^{(i)})\geq k(\alpha_j)^{1/2}$ we obtain \eqref{C} from
\[
 |B_{i,j}(\alpha_j)|=\frac{|B_i(\alpha_j)|}{\rho(\alpha_j, \lambda_j^{(i)})}\leq k(\alpha_j)^{1/2}\ .
\]
If $\rho(\alpha_j, \lambda_j^{(i)})\leq k(\alpha_j)^{1/2}$ we use Lemma~\ref{bounded-f}(c) to see that
\[
\left|(1-|\alpha_j|)| B_i^\prime(\alpha_j)|- (1-|\lambda_j^{(i)}|)|B_i^\prime(\lambda_j^{(i)})|\right|\leq 6 \rho(\alpha_j, \lambda_j^{(i)})\ .
\]
Since $(1-|\alpha_j|)| B_i^\prime(\alpha_j)|\leq k(\alpha_j)$ we deduce that $(1-|\lambda_j^{(i)}|)|B_i^\prime(\lambda_j^{(i)})|\leq 7 k(\alpha_j)^{1/2}$, that is,
$
 |B_{i,j}(\lambda_j^{(i)})|\leq 7 k(\alpha_j)^{1/2}\ .
$
Since $\rho(\alpha_j,\lambda_j^{(i)})\leq k(\alpha_j)^{1/2}$, by Schwarz's lemma 
we deduce that $|B_{i,j}(\alpha_j)|\leq C_1 k(\alpha_j)^{1/2}$ for some $C_1>0$ and \eqref{C} holds also in this case.

\medskip

For $i=1,2$ and $j\in \N$ let
\[
 E_{i,j}=\{z\in\D : B_i(z)=0\, \textrm{and}\, \rho(z,\alpha_j) < 1/2 \}\ ,
\]
and let $b_{i,j}$ be the Blaschke product with zeros in $E_{i,j}$ so that $B_i=b_{i,j} B_i(j)$. 
Since $|B_i(\alpha_j)|\leq k(\alpha_j)$, estimates \eqref{B}, \eqref{k100} and \eqref{C} give
\begin{align}\label{D}
  & |b_{i,j}(\alpha_j)|\leq  k(\alpha_j) e^{H_0(\alpha_j)} \leq  k(\alpha_j)^{99/100}\ ,\\
  & \prod_{\stackrel{z\in E_{i,j}}{z\neq\lambda_j^{(i)}}} \rho(z,\alpha_j)
 =\frac{|B_{i,j}(\alpha_j)|}{|B_i(j)(\alpha_j)|}\leq C k(\alpha_j)^{1/2} e^{H_0(\alpha_j)} \leq C k(\alpha_j)^{1/2-1/100}\ . 
\end{align}

In order to prove \eqref{A} we will now split $\{\alpha_j\}_j$ into different pieces and consider
different cases according to the number of zeros of $B_1$ and $B_2$ in the following
neighborhoods of $\alpha_j$: $U_j=D(\alpha_j,  k(\alpha_j)^{1/10})$ and
$\tilde{U}_j=D(\alpha_j,  k(\alpha_j)^{1/100})\supset U_j$. Here are the cases we are going to discuss
now:
\begin{itemize}
\item [(i)] At least one Blaschke product has at least two zeros in $U_j$. The set of these
$\alpha_j$ will be denoted by $A_1$. Splitting possibly $A_1$ into two subsequences we can 
assume that $B_1$ has at least two zeros in $U_j$ (in case $B_2$ has at least two zeros in $U_j$ while
$B_1$ has not, inversing the r\^oles of $B_1$ and $B_2$ yields the exact same estimate). In this case
we will distinguish three  subcases.
\begin{itemize}
\item [(i)-a.] $B_2$ has at least two zeros in $\tilde{U}_j$. The set of these $\alpha_j$ will be
denoted by $A_{11}$.
\item [(i)-b.] $B_2$ has no zero in $\tilde{U}_j$. The set of these $\alpha_j$ will be
denoted by $A_{12}$.
\item [(i)-c.] $B_2$ has exactly one zero in $\tilde{U}_j$. The set of these $\alpha_j$ will be
denoted by $A_{13}$.
\end{itemize}
\item [(ii)] Both Blaschke products have at most one zero in $U_j$. The set of these $\alpha_j$ will be denoted by $A_2$. 
\end{itemize}
We will establish \eqref{A} in each of these cases.

{\bf Case (i)-a.} We will start with
 $\alpha_j\in A_{11}$. For $i=1,2$ pick two zeros of $b_{i,j}$ in $\tilde U_j$ and let 
$\tilde{b}_{i,j}$ be the corresponding Blaschke product of degree 2.
Consider $\tilde B_i=\prod_j \tilde{b}_{i,j}$ where the product is taken over all $j$ such that $\alpha_j\in A_{11}$. Since $\tilde B_i$ is a subproduct of $B_i$, the assumption (d) and Lemma~\ref{ffgg} give $I(\tilde B_1, \tilde B_2)=J(\tilde B_1,\tilde B_2)$.
Applying Lemma~\ref{45} with $\eps_j=k(\alpha_j)^{1/100}$ we obtain $H\in\Har_+(\D)$ such that
\[
 k(\alpha_j)^{1/100}\geq e^{-H(\alpha_j)},\quad \alpha_j\in A_{11}\ .
\]
This gives the required estimate \eqref{A} for the points in $A_{11}$.

{\bf Case (i)-b.}
The idea in this case is to replace $B_2$ by an appropriate perturbation $B_2-G\tilde{B}_1$, where $\tilde{B}_1$ is 
a sub-product of $B_1$ vanishing exactly twice in each $\tilde{U}_j$,  in order to generate two zeros (controlled
by Rouch\'e'e theorem) and then conclude as in Case (ii)-a.

For $\alpha_j\in A_{12}$ the function $b_{2,j}$ has no zero in $\tilde U_j$. 
For each $\alpha_j\in A_{12}$ pick two zeros of $B_1$ in $U_j$ and let $\tilde B_1$ be the Blaschke product formed with these zeros as in case (i)-a. 
Since $U_j\subset \mathcal{D}_j$ and the disks $\mathcal{D}_j$
are disjoint, $\tilde{B}_1$ is a Blaschke product whose zeros
form a union of two Nevanlinna interpolating sequences \cite[Corollary 1.9]{HMNT}.
Hence there exists $H\in\Har_+(\D)$ such that for every zero $\lambda\in U_j$ of $\tilde{B}_1$,
and $z$ with $\rho(z,\alpha_j)=k(\alpha_j)^{1/30}$,
\bea\label{B1tilde}
 |\tilde B_1(z)|\ge 
  e^{-H(z)}\rho(z,\lambda)\ge e^{-H(z)}\dist(z,\partial U_j)
 \ge k(\alpha_j)^{1/15} e^{-H(z)}\ .
\eea
Let $G=e^{H+i\tilde H}$, where $\tilde H$ is the harmonic conjugate of $H$. By Lemma~\ref{ffgg}, $I(\tilde B_1,B_2)=J(\tilde B_1,B_2)$. Then, observing that $G$ is invertible in $N$, one has
\[
 I(\tilde{B}_1,B_2)=I(\tilde B_1, G\tilde B_1-B_2) \subset J(\tilde B_1, G\tilde B_1-B_2) 
 \subset J(\tilde{B}_1,B_2)=I(\tilde{B}_1,B_2)\ ,
\]
hence
\[
 I(\tilde B_1, G\tilde B_1-B_2) = J(\tilde B_1, G\tilde B_1-B_2) \ .
\]
Now, for points $z\in \D$ such that $\rho(z,\alpha_j)=k(\alpha_j)^{1/30}$ we have, by Lemma~\ref{Blaschke-no-zeros} and the assumption:
\begin{align}\label{Rouche}
 \left|\tilde B_1(z) G(z)-(\tilde B_1(z) G(z)-B_2(z))\right|=|B_2(z)|\leq
 |B_2(\alpha_j)|^{\frac{1-\rho_j}{1+\rho_j}}
 \leq k(\alpha_j)^{\frac{1-\rho_j}{1+\rho_j}}\ ,
\end{align}
where $\rho_j = \rho_{\tilde U_j} (z, \alpha_j)$ and $\rho_{\tilde U_j}$ indicates the pseudohyperbolic distance in $\tilde U_j$.  
Since $\rho(z,\alpha_j)=k(\alpha_j)^{1/30}$ and $\tilde U_j=D(\alpha_j, k(\alpha_j)^{1/100})
\supset U_j$, we have 
\[
 \rho_{\tilde U_j}(z,\alpha_j)\leq k(\alpha_j)^{1/30- 1/100}\ .
\]
Indeed we can assume $\alpha_j=0$ and let $\phi:\D\longrightarrow \tilde U_j$ be given by $\phi(w)=k(\alpha_j)^{1/100} w$; then 
\[
 \rho_{\tilde U_j}(z,\alpha_j)=\rho(\frac{z}{k(\alpha_j)^{1/100}},0)=\frac{|z|}{k(\alpha_j)^{1/100}}\leq k(\alpha_j)^{1/30 - 1/100}\ .
\]
Since we can assume that $k(\alpha_j)$ is small, say $k(\alpha_j)^{1/30 - 1/100}<\eps$, 
we deduce from \eqref{Rouche} and \eqref{B1tilde} that
\[
 \left|\tilde B_1(z) G(z)-(\tilde B_1(z) G(z)-B_2(z))\right|\leq k(\alpha_j)^{\frac{1-\eps}{1+\eps}}<k(\alpha_j)^{1/15}<|\tilde B_1(z) G(z)|\ .
\]

Then, by Rouch\'e's theorem $\tilde B_1 G-B_2$ has two zeros in $D(\alpha_j, k(\alpha_j)^{1/30})$.
Observe that we can replace $\tilde{B}_1G-B_2$ by the Blaschke product vanishing on the zeros
of $\tilde{B}_1G-B_2$, and we can thus argue as we have done for $A_{11}$ (note that $k(\alpha_j)$ now 
only gives the size of $\mathcal{D}_j$, $U_j$ and $\tilde{U}_j$, and it only depends on the fact
that the Blaschke products under consideration have zeros in these neighborhoods, but not on the
explicit form of these products).

{\bf Case (i)-c.}
Recall that $A_{13}$ is the set of $\alpha_j\in A_{1}$ such that $b_{2,j}$ has one zero in $\tilde U_j$. If $\alpha_j\in A_{13}$, the zero set of $b_{2,j}$ in $\tilde U_j$ must be $\lambda_j^{(2)}$. Recall from \eqref{C} that
\[
 |B_{2,j}(\alpha_j)|\le C k(\alpha_j)^{1/2}
\]
and $B_{2,j}$ has no zeros in $\tilde U_j$. Hence by \eqref{B} and \eqref{k100}, we deduce that
\[
 \Bigl|\frac{B_2(\alpha_j)}{\prod\limits_{\alpha_k\in A_{13}} b_{\lambda_k^{(2)}}(\alpha_j)}\Bigr|
 =\Bigl|\frac{B_{2,j}(\alpha_j)}{\prod\limits_{\stackrel{k\neq j}{\alpha_k\in A_{13}}}b_{\lambda_k^{(2)}}(\alpha_j)}\Bigr|
 \le C k(\alpha_j)^{1/2} e^{H_0(\alpha_j)}\leq C k(\alpha_j)^{0.49}.
\]
Thus, replacing $B_2$ by $B_2/\prod\limits_{\alpha_k\in A_{13}} b_{\lambda_j^{(2)}}$  we can assume that $B_2$ has no zeros in $\tilde{U_j}$ and we can argue as in the previous case.

{\bf Case (ii).}
For $\alpha_j\in A_2$ and $i=1,2$ the function $b_{i,j}$ has at most one zero in $U_j$. If it has
one zero, this must actually be $\lambda_j^{(i)}$. 
In this case, from \eqref{B} and \eqref{C},
\bea\label{bij}
\left|\frac{b_{i,j}(\alpha_j)}{b_{\lambda_j^{(i)}}(\alpha_j)}\right| 
 = \left|\frac{B_{i,j}(\alpha_j)}{B_i(j)(\alpha_j)}\right|\le Ck(\alpha_j)^{1/2}e^{H_0(\alpha_j)}
 \le Ck(\alpha_j)^{0.49}.
\eea
Hence, replacing $b_{i,j}$ by $b_{i,j}/b_{\lambda_j^{(i)}}$ we can assume that $b_{i,j}$ has no zeros in $U_j$ and satisfies the above estimate \eqref{bij}. Observe from this estimate that the initial zero-set 
$E_{i,j}$ cannot be reduced to the sole point $\lambda_j^{(i)}$. We will henceforth assume that
$E_{i,j}$ does not contain any point in $U_j$.
In order to apply Lemma~\ref{41} write $E_{i,j} = \{a_k (i,j) : k=1, \ldots , N \}$, where the points $a_k (i,j)$ are taken so that the corresponding distances $m_k= m_k (i,j) = \rho(a_k (i,j),\alpha_j)$ satisfy 
$m_1\le m_2 \le\cdots \le m_N$. In particular $a_1 (i,j)$ is the closest point of
$E_{i,j}$ to $\alpha_j$, and it is outside $U_j$, that is,  $\rho(a_1 (i,j),\alpha_j)\ge k(\alpha_j)^{1/10}$. 
According to \eqref{bij}, and setting $\eta=Ck(\alpha_j)^{0.49}$, we have
\[
 \prod_{k=1}^N m_k=|b_{i,j}(\alpha_j)|\le Ck(\alpha_j)^{0.49}=\eta.
\]
Moreover
\[
 \prod_{k=2}^N m_k=\frac{|b_{i,j}(\alpha_j)|}{\rho(a_1 (i,j),\alpha_j)}\leq C k(\alpha_j)^{0.39}
 \le \eta^{1/2}\ ,
\]
when $k(\alpha_j)$ is sufficiently small (which we can assume). We are thus in the conditions
of Lemma~\ref{41}, which allows to split the product $b_{i,j}$ into two sub-products, denoted by 
$b_{i,j}^*$, $b_{i,j}^{**}$ each of which is
controlled by $\eta^{1/4}=C^{1/4}k(\alpha_j)^{0.49/4}$. More concretely
\[
 b_{i,j}=b_{i,j}^* b_{i,j}^{**},\qquad j\in\N\ , \ i=1,2\ ,
\]
and 
\bea\label{bij**estim}
 |b_{i,j}^*(\alpha_j)|\leq  |b_{i,j}^{**}(\alpha_j)|\leq C_1 k(\alpha_j)^{0.1225}
\eea
(if the first inequality does not hold interchange the roles of $b_{i,j}^*$ and $b_{i,j}^{**}$).
Let 
\[
 B_i^*=\prod_{j=1}^\infty\ b_{i,j}^*,\qquad B_i^{**}=\prod_{j=1}^\infty\ b_{i,j}^{**}\ ,
\]
where the product is taken over the indices $j$ such that $\alpha_j\in A_2$. For $j\in\N$ and $i=1,2$ we have
\begin{equation}\label{F}
|B_i^*(\alpha_j)|+ |B_i^{**}(\alpha_j)|\leq 2C_1 k(\alpha_j)^{0.1225}\ .
\end{equation}
Moreover, taking into account \eqref{B}, there exists $H_0\in\Har_+(\D)$ such that
\begin{equation}\label{G}
 |B_i^*(\alpha_j)|\leq e^{H_0(\alpha_j)} |B_i^{**}(\alpha_j)|\ .
\end{equation}
Split $A_2$ into two sequences $A_2=A_{21}\cup A_{22}$, where 
\[
A_{21}=\{\alpha_j : |B_1^{**}(\alpha_j)|\leq |B_2^{**}(\alpha_j)|\}\ , \qquad
A_{22}=\{\alpha_j : |B_2^{**}(\alpha_j)| < |B_1^{**}(\alpha_j)|\} \ .
\]
For $i=1,2$ we will construct $H_{2i}\in\Har_+(\D)$ such that $k(\alpha_j)\geq e^{-H_{2i}(\alpha_j)}$ for any $\alpha_j\in A_{2i}$. This will give \eqref{A} also in this case and finish the proof. Let us explain how to construct $H_{21}$. The same argument applies to $H_{22}$.
For $\alpha_j\in A_{21}$ pick $\alpha_j^*\in \D$ with $\rho(\alpha_j^*,\alpha_j)=|B_1^{**}(\alpha_j)|/4$. Observe that \eqref{G} yields
\bea\label{fracB1}
 \frac{|B_1(\alpha_j)|}{|B_2^{**}(\alpha_j)|} \leq |B_1^{*}(\alpha_j)|\frac{|B_1^{**}(\alpha_j)|}{|B_2^{**}(\alpha_j)|}\leq |B_1^{*}(\alpha_j)|
 \leq e^{H_0(\alpha_j)} 4 \rho(\alpha_j^*,\alpha_j)\ .
\eea
Since $|B_1^{**}(\alpha_j)|\leq |B_2^{**}(\alpha_j)|$ and $\rho(\alpha_j^*,\alpha_j)=|B_1^{**}(\alpha_j)|/4$, Schwarz' Lemma (see Lemma \ref{bounded-f}(a)) gives $|B_2^{**}(\alpha_j^*)|\geq  |B_2^{**}(\alpha_j)| / 2$ and $|B_1^{**}(\alpha_j^*)|\leq  3|B_1^{**}(\alpha_j)| / 2$. Hence, using
again Lemma \ref{bounded-f}(a) and \eqref{G},
\begin{align}\label{fracB2}
 \frac{|B_1(\alpha_j^*)|}{|B_2^{**}(\alpha_j^*)|}&\leq|B_1^{*}(\alpha_j^*)|\frac{|B_1^{**}(\alpha_j^*)|}{|B_2^{**}(\alpha_j^*)|}\leq 3 |B_1^{*}(\alpha_j^*)|
 \leq 3 \bigl(|B_1^{*}(\alpha_j)|+|B_1^{*}(\alpha_j^*)-B_1^{*}(\alpha_j)|)\nonumber\\
 &\leq 3 \bigl(|B_1^{*}(\alpha_j)|+2\rho(\alpha_j^*,\alpha_j)|)
   = 3 \bigl(|B_1^{*}(\alpha_j)|+\frac{|B_1^{**}(\alpha_j)|}{2}\bigr)\nonumber\\
 &\leq C_2 e^{H_0(\alpha_j)} |B_1^{**}(\alpha_j)|= 4 C_2 e^{H_0(\alpha_j)} \rho(\alpha_j^*,\alpha_j)\ ,
\end{align}
where $C_2 >0$ is an absolute constant. From \eqref{fracB1} and \eqref{fracB2} we get
\[
 \left|\frac{B_1(\alpha_j^*)}{B_2^{**}(\alpha_j^*)}-\frac{B_1(\alpha_j)}{B_2^{**}(\alpha_j)}\right|
 \le \frac{|B_1(\alpha_j^*)|}{|B_2^{**}(\alpha_j^*)|}+\frac{|B_1(\alpha_j)|}{|B_2^{**}(\alpha_j)|}
 \le 4(C_2+1) e^{H_0(\alpha_j)} \rho(\alpha_j^*,\alpha_j)
\]
Hence the sequence defined by $w(\alpha_j)=B_1(\alpha_j)/B_2^{**}(\alpha_j)$ and
$w(\alpha_j^*)=B_1(\alpha_j^*)/B_2^{**}(\alpha_j^*)$ is in the trace space defined 
on the sequence $\{\alpha_j,\alpha_j^*\}_{\alpha_j\in A_{21}}$, and
according to Lemma~\ref{diff-div} we find $h\in N$ such that
\[
 h(\alpha_j)=\frac{B_1(\alpha_j)}{B_2^{**}(\alpha_j)}\ ,\qquad h(\alpha_j^*)=\frac{B_1(\alpha_j^*)}{B_2^{**}(\alpha_j^*)}\ ,\qquad \alpha_j\in A_{21}\ .
\]
Setting $b$ the Blaschke product with zeros $\alpha_j\in A_{21}$ and $b^*$ the Blaschke product with zeros $\alpha_j^*$,  we thus get $g\in N$ such that 
\begin{equation}\label{H}
B_1=B_2^{**} h + bb^* g\, .
\end{equation}
Since $I(B_1,B_2)=J(B_1,B_2)$, Lemma~\ref{ffgg} yields $ I(B_1,B_2^{**})=J(B_1,B_2^{**})$.
Now \eqref{H} gives also $I(B_1,B_2^{**})=I(bb^*g, B_2^{**})$ and  $J(B_1,B_2^{**})=J(bb^*g, B_2^{**})$. Hence, 
$I(bb^*g, B_2^{**})=J(bb^*g, B_2^{**})$, and again by Lemma~\ref{ffgg},
$
 I(bb^*, B_2^{**})=J(bb^*, B_2^{**})\ 
$
(observe that $bb^*g$ and $B_2^{**}$ --- which is a subproduct of $B_2$ --- have no common zeros, 
since by \eqref{H} those common zeros would be in common with $B_1$, which we excluded).

Now notice that $bb^*$ has two zeros in $D(\alpha_j, |B_2^{**}(\alpha_j^*)|/2)$.
Also, from \eqref{bij**estim}
we can deduce that $D(\alpha_j, |B_2^{**}(\alpha_j^*)|/2)\subset D(\alpha_j, C_1 k(\alpha_j)^{0.1})$.
Hence we are in the same situation as we were discussing for $A_1$, now applied to
$(B_2^{**},bb^*)$, and therefore there exists $H_{21}\in\Har_+(\D)$ such that
\bea\label{estimA21}
\quad |B_2^{**}(\alpha_j)|+(1-|\alpha_j|) \Big[|(B_2^{**})'(\alpha_j)|+|(bb^*)'(\alpha_j)|\Big]\geq e^{-H_{21}(\alpha_k)}\ , \ \alpha_j\in A_{21}\ .
\eea
Now, by \eqref{bij**estim}, $ |B_2^{**}(\alpha_j)|\leq 2 C_1  k(\alpha_j)^{0.1225}$.
Also $B_2^{**}$ has no zeros in $U_j$, and so Lemma~\ref{Blaschke-no-zeros}(b) gives that 
\[
 (1-|\alpha_j|) |(B_2^{**})'(\alpha_j)|
 \leq \dfrac{|B_2^{**} (\alpha_j)|}{k(\alpha_j)^{1/10}} \log |B_2^{**} (\alpha_j)|^{-2} \lesssim k(\alpha_j)^{0.02} , \ \alpha_j \in A_{21} 
\]
Moreover, since $(bb^*)' (\alpha_j) = b'(\alpha_j)b^*(\alpha_j)$ and 
$(1-|\alpha_j|)|b'(\alpha_j) | \leq 2$, we get
\[
 (1-|\alpha_j|)|(bb^*)'(\alpha_j)|
 \leq 2 \rho(\alpha_j , {\alpha}_j^*) \leq |B_1^{**} (\alpha_j)| /2 
\]
and, again by \eqref{bij**estim}, this expression is controlled by $C_1k(\alpha_j)^{0.1225}$.
As a result, there exists an absolute constant $C_3>0$ such that the left hand side of \eqref{estimA21} is upper bounded by $C_3 k(\alpha_j)^{0.02}$, and we deduce that
\[
 C_3 k(\alpha_j)^{0.02} \geq e^{-H(\alpha_j)}\ , \alpha_j\in A_{21}\ ,
\]
as desired.\hfill\qedsymbol

Finally let us show that {when $m \geq 3$,} condition (d) does not imply the equivalent conditions (a), (b) or (c) in Theorem \ref{MainThm}. The example is analogous to the one given in the context of $H^\infty $ in \cite{GMN}. Let $B_1$, $B_2$ be Nevanlinna interpolating Blaschke products with zero sets $\Lambda_1$ and $\Lambda_2$. We first claim that 
\bea\label{tres}
I(B_1^2 , B_2^2 , B_1 B_2) = J(B_1^2 , B_2^2 , B_1 B_2) \, .
\eea
To prove this we can assume that $B_1$ and $B_2$ have no common zeros. Let $f \in J(B_1^2 , B_2^2 , B_1 B_2)$. Then there exists $H \in \Har_+ (\D)$ such that 
\begin{eqnarray}\label{counterex}
|f(z)| \leq e^{H(z)} (|B_1(z)|^2 + |B_2(z)|^2 + |B_1(z)B_2 (z)| ), \quad z \in \D.
\end{eqnarray} 
Then $|f(\lambda)|\le e^{H(\lambda)}|B_1(\lambda)|^2$ for $\lambda\in \Lambda_2$ so that there exists
$g_1 \in N$ with $g_1 (\lambda) = f(\lambda)/ B_1^2 (\lambda)  $, $\lambda \in \Lambda_2$. This implies that there is $g_2 \in N$ such that $f= g_1 B_1^2 + B_2 g_2$. Observe that for every $\lambda \in \Lambda_1$ we have $|g_2 (\lambda)|/|B_2(\lambda)| = |f(\lambda)|/ |B_2 (\lambda)|^2 $ which by
\eqref{counterex} is bounded by $e^{H(\lambda)} $. Hence there exists $g_3 \in N$ with $g_3 (\lambda) = g_2 (\lambda) / B_2 (\lambda) $, $\lambda \in \Lambda_1$. Hence, there exists $g_4 \in N$ with $g_2 = B_2 g_3 + B_1 g_4 $. Finally, $f=B_1^2 g_1 + B_2^2 g_3 + B_1 B_2 g_4$ and $f \in I(B_1^2 , B_2^2 , B_1 B_2)$. Hence \eqref{tres} holds. However, if the sequences $\Lambda_1$ and $\Lambda_2$ are too close, then using condition (c) of Theorem 1.1 it can be seen that the ideal $I(B_1^2, B_2^2 , B_1 B_2)$ cannot contain a Nevanlinna interpolating Blaschke product.


\section{Two open problems}\label{S4}

\subsection{The stable rank of the Nevanlinna class}

The first open problem we discuss concerns the stable rank of the Nevanlinna algebra.
Recall that an $m$-tuple $(a_1,\ldots,a_m)$ of elements of a commutative unital algebra $A$ is called \emph{unimodular} if the ideal it generates is the whole algebra, that is, \ there exists an $m$-tuple $(b_1,\ldots, b_m)$ in $A^m$ such that $\sum_{i=1}^m a_ib_i=1$. The $m$-tuple $(a_1,\ldots,a_m)$ is called \emph{reducible} if 
there exists an $(m-1)$-tuple $(x_1,\ldots,x_{m-1})$ in $A^{m-1}$ such that $(a_1+x_1 a_m,\ldots,a_{m-1}+x_{m-1}a_m)$ is 
unimodular (so, the ideal  generated by $(a_1,\ldots, a_m)$
contains a specific $(m-1)$-tuple that already generates $A$).
The \emph{stable rank} of the algebra is the least $m$ for which
every unimodular $m+1$-tuple is reducible.

It is known that the stable rank of the disk algebra and of $H^{\infty}$ is equal to one (see 
\cite{CS} or \cite{JMW} for the disk algebra and \cite{Tr} for $H^{\infty}$).  The stable rank for the Nevanlinna class
is unknown, but the following result shows that it is at least two.

\begin{proposition}\label{rkge2}
The stable rank of the Nevanlinna class is at least 2.
\end{proposition}

It is worth mentioning that any triple $(f_1 , f_2 , f_3) \in N^3$ such that for some $i$ the zeros of $f_i$ form a Nevanlinna interpolating sequence, can be reduced. The argument uses  Theorem \ref{IntThm}, but it is lenghty and we do not include the details here. 
\\
\\
{\bf Open problem:} Is the stable rank of $N$ equal to $2$?

\begin{proof}[Proof of Proposition \ref{rkge2}]
Suppose to the contrary that the stable rank of $N$ is one and let us reach a contradiction.
For any unimodular pair of Blaschke products, 
there will then exist $\Phi_1\in N$ such that $B_1+\Phi_1 B_2$ 
is invertible in $N$, i.e.\ 
\bea\label{rk1}
 B_1+\Phi_1 B_2=e^f\, ,
\eea 
where 
$\Re(f)=H_+-H_-$, for some $H_+,H_+\in\Har_+(\D)$. 
 We will show that this is not possible in general.
To this end, let $\Lambda_1=\{\lambda_n\}_n:=\{1-2^{-n}\}_n$ and $B_1$ the associated Blaschke product. 
The sequence $\Lambda_1$ is $\Hi$-interpolating. Take now $\{\mu_n\}_n\subset (0,1)$ with $\rho(\lambda_n,\mu_n)$ small enough so that
\bea\label{IC}
 |B_1(\mu_n)|=\left\{
 \begin{array}{ll}
  e^{-\frac{1}{1-|\lambda_n|}} &\text{if $n$ even}\\
  e^{-\frac{2}{1-|\lambda_n|}} &\text{if $n$ odd}.
 \end{array}
 \right.
\eea
Set $\Lambda_2=\{\mu_n\}_n$ and $B_2$ its Blaschke product. 

We shall see first  that $(B_1, B_2)$ is unimodular, i.e, that there exists $H\in\Har_+(\D)$ such that 
$
 |B_1|+|B_2|\geq e^{-H}\ .
$
Fix a $\delta>0$ such that the regions $\Omega_n=D(\lambda_n,\delta)\cup D(\mu_n,\delta)$ are mutually disjoint. Since $\Lambda_1$ and $\Lambda_2$ are $H^\infty$-interpolating sequences,
there exists $\eta>0$ such that
\[
|B_i(z)|\ge \eta\ ,\ z \in \D \setminus \cup_n \Omega_n, \qquad i=1,2.
\]
Thus we only need to care about the estimate on $D(\lambda,\delta)$,  for $\lambda\in\Lambda_1\cup 
\Lambda_2$.  So suppose $\lambda=\lambda_n$ or $\lambda=\mu_n$. Since $|B_1/b_{\lambda_n}|$ and $|B_2/b_{\mu_n}|$ are bounded below on $D(\lambda , \delta)$ (by Carleson's condition), we only need to take care of $|b_{\lambda_n}(z)| + |b_{\mu_n}(z)|$. By  \eqref{IC}
\[
 \rho(\lambda_n,\mu_n)=|b_{\lambda_n}(\mu_n)|\ge  e^{-\frac{2}{1-|\lambda_n|}}\ .
\]
By the triangular inequality
$
 \left| \rho(\lambda_n,z)- \rho(z,\mu_n)\right|\leq \rho(\lambda_n,\mu_n)\ ,
$ thus either $\rho(\lambda_n,z)$ or $\rho(\mu_n,z)$ are greater than $(1/2)e^{-\frac{2}{1-|\lambda_n|}}$. Take now $c$ (independent of $n$) such that for $z\in D(\lambda,\delta)$,
\[
(1/2) e^{-\frac{2}{1-|\lambda_n|}}\geq e^{-\frac{c}{1-|z|}}\ .
\]
With this
\[
|B_1(z)|+|B_2(z)|\geq e^{-\frac{c}{1-|z|}}\ , \qquad z\in D(\lambda,\delta)\ .
\]
Since 
\[
H_0(z)=\Re(\frac{1+z}{1-z})=\frac{1-|z|^2}{|1-z|^2}\in\Har_+(\D)
\]
and
\[
 H_0(z)\asymp\frac 1{1-|z|}\ ,\qquad z\in D(\lambda,\delta)
\]
this finally implies that $(B_1, B_2)$ is unimodular.

Let us now show that the pair $(B_1,B_2)$ cannot be reduced.
Equation \eqref{rk1} on $\mu_n$ yields
\[
 \log |B_1(\mu_n)|=H_+(\mu_n)-H_-(\mu_n)=P[\nu](\mu_n)\ ,\quad n\in\N,
\]
where $\nu$ is a finite measure on $\partial\D$ such that $\Re(f)=P[\nu]$. Then, since $\{\mu_n\}_n$ tends
radially towards $1$,
\[
 \lim_{n\to\infty} (1-|\mu_n|^2)P[\nu](\mu_n)=\nu(\{1\})\ .
\]
But from \eqref{IC} we see that 
$\{(1-|\mu_n|^2)\log|B_1(\mu_n)|\}_n$ 
has no limit, so we have reached a contradiction.
\end{proof}

\subsection{The $f^2$ problem}

In the late seventies T. Wolff presented a problem on ideals of $H^\infty$, known now as the $f^2$ problem,  which was finally solved by S. Treil in \cite{Tr2}. We now discuss an analogous problem in the Nevanlinna class. Let $f_1,\ldots,f_n$ be  functions in the Nevanlinna class, and let $f\in N$ be such that there exists
$H\in \Har_+(\D)$ with
\bea\label{cond}
 |f(z)|\le e^{H(z)}(|f_1(z)|+\cdots |f_n(z)|)^p, \quad z\in \D,
\eea
for some $p\ge 1$. Does it follow that $f\in I(f_1,\ldots,f_n)$ ?

As in the $H^{\infty}$ case, when $p>2$, the $\overline{\partial}$ estimates by T. Wolff show
that the answer is affirmative. When $p<2$ the answer is in general negative, as the following example shows.
Let $N$ be an integer such that $N+1>2Np$, $f=B_1^NB_2^N$, $f_1=B_1^{N+1}$ and
$f_2=B_2^{N+1}$. Then \eqref{cond} holds but $f\notin I(f_1,f_2)$ if $(B_1,B_2)$ is not
unimodular in $N$.

{\bf Open problem:} What happens in the case  $p=2$?


\begin{thebibliography}{BRSHZE}

\bibitem{Bourgain}
Bourgain, J.
\emph{On finitely generated closed ideals in $H^{\infty}(\D)$}.
Ann. Inst. Fourier (35), n$^\circ$ 4 (1985), 163--174.

\bibitem{Carl}
Carleson, L.
\emph{An interpolation problem for bounded analytic functions}. 
Amer. J. Math. 80 (1958), 921--930. 

\bibitem{CS}
 Corach, G., Su\'arez, D.
\emph{On the stable range of uniform algebras and $H^\infty$}. 
 Proc. Amer. Math. Soc. 98 (1986), no. 4, 607--610.
 
 \bibitem{DGM}
 Daepp, U., Gorkin, P., Mortini, R.
\emph{Finitely generated radical ideals in  $H^\infty$}. 
 Proc. Amer. Math. Soc. 116 (1992), no. 2, 483--488.

 

\bibitem{Garn}
Garnett, John B.
\emph{Bounded analytic functions}. 
Revised first edition. Graduate Texts in Mathematics, 236. Springer, New York, 2007. xiv+459 pp.

\bibitem{GGJ}
Garnett, J. B., Gehring, F. W. , Jones, P. W. 
\emph{Conformally invariant length sums}. 
Indiana Univ. Math. J. 32 (1983), no. 6, 809--829.

 \bibitem{GIM}
Gorkin, P., Izuchi, K., Mortini, R. 
\emph{ Higher order hulls in $H^\infty$ II}. 
 J. Funct. Anal. 177 (2000), no. 1, 107--129.


\bibitem{GM1}
Gorkin, P., Mortini, R. 
\emph{ Countably generated prime ideals in $H^\infty$}. 
Math. Z. 251 (2005), no. 3, 523--533. 


\bibitem{GM3}
Gorkin, P., Mortini, R. 
\emph{  Hulls of closed prime ideals in  $H^\infty$}. 
Illinois J. Math. 46 (2002), no. 2, 519--532.
 
\bibitem{GM4}
Gorkin, P., Mortini, R. 
\emph{ Alling's conjecture on closed prime ideals in  $H^\infty$}. 
 J. Funct. Anal. 148 (1997), no. 1, 185-–190.

\bibitem{GMN}
Gorkin, P., Mortini, R. , Nicolau, A.
\emph{The generalized corona theorem}. 
Math. Ann. 301 (1995), no. 1, 135--154. 

\bibitem{HMN}
Hartmann, A., Massaneda, X., Nicolau, A.
\emph{Traces of the Nevanlinna class on discrete sequences}. 
In preparation.

\bibitem{HMNT}
Hartmann, A., Massaneda, X., Nicolau, A., Thomas, P.J. 
\emph{Interpolation in the Nevanlinna and Smirnov classes and harmonic majorants}. 
J. Funct. Anal. 217 (2004), no. 1, 1--37.

\bibitem{JMW}
Jones, P. W.; Marshall, D.; Wolff, T. 
\emph{Stable rank of the disc algebra}. Proc. Amer. Math. Soc. 96 (1986), no. 4, 603--604.

\bibitem{KL}
Kerr-Lawson A., 
\emph{Some lemmas on interpolating Blaschke products and correction}, 
Canad. J. Math., 21 (1969), pp. 531--534.

\bibitem{Mart}
Martin, R. 
\emph{On the ideal structure of the Nevanlinna class}. 
Proc. Amer. Math. Soc. 114 (1992), no. 1, 135--143.

\bibitem{Mort0}
Mortini, R. 
\emph{Zur Idealstruktur von Unterringen der Nevanlinna-Klasse N. (German) [On the ideal structure of subrings of Nevanlinna class N] }. 
Publ. du CUL, Séminaire de Math. de Luxembourg, Travaux mathématiques I (1989), 81--91



\bibitem{Mort1}
Mortini, R. 
\emph{Generating sets for ideals of finite type in $H^{\infty}$}. 
Bull. Sci. Math. 136 (2012), no. 6, 687--708.

\bibitem{Mort2}
Mortini, R. 
\emph{Finitely generated closed ideals in  $H^{\infty}$}. 
Arch. Math. (Basel) 45 (1985), no. 6, 546--548.


\bibitem{Naft}
Naftalevi\v c, A.G.,
\emph{On interpolation by functions of
bounded characteristic (Russian)},  Vilniaus Valst. Univ. Moksl\c u Darbai.
Mat. Fiz. Chem. Moksl\c u  {Ser. 5} (1956), 5--27.

\bibitem{NikTr}
Nikolski, N.K.,
\emph{Treatise on the shift operator. Spectral function theory. With an appendix by S. V. Hruščev [S. V. Khrushchëv] and V. V. Peller.}
Translated from the Russian by Jaak Peetre. Grundlehren der Mathematischen Wissenschaften [Fundamental Principles of Mathematical Sciences], 273. Springer-Verlag, Berlin, 1986


\bibitem{OS}
Ortega-Cerd\`a, J., Seip, K. 
\emph {Harmonic measure and uniform densities}. 
Indiana Univ. Math. J. 53 (2004), no. 3, 905--923.

\bibitem{ShSh}
Shapiro, J. H.; Shields, A. L. 
{\it Unusual topological properties of the Nevanlinna class}. 
Amer. J. Math. 97 (1975), no. 4, 915--936.

\bibitem{Tol83}
Tolokonnikov, V. A. 
\emph{Interpolating Blaschke products and ideals of the algebra $H^{\infty}$}. 
(Russian) Investigations on linear operators and the theory of functions, XII. Zap. Nauchn. Sem. Leningrad. Otdel. Mat. Inst. Steklov. (LOMI) 126 (1983), 196--201.

\bibitem{Tol88}
Tolokonnikov, V. A. 
\emph{Blaschke products with the Carleson-Newman condition, and ideals of the algebra 
$H^{\infty}$}. 
(Russian) Zap. Nauchn. Sem. Leningrad. Otdel. Mat. Inst. Steklov. (LOMI) 149 (1986), Issled. Linein. Teor. Funktsii. XV, 93--102, 188; translation in J. Soviet Math. 42 (1988), no. 2, 1603--1610 


\bibitem{Tr}
Treil, S.
\emph{The stable rank of the algebra $H^{\infty}$ equals 1}.
J. Funct. Anal. 109 (1992), no. 1, 130--154.

\bibitem{Tr2}
Treil, S.
\emph{Estimates in the corona theorem and ideals of $H^{\infty}$: 
a problem of T. Wolff}. 
Dedicated to the memory of Thomas H. Wolff. 
J. Anal. Math. 87 (2002), 481--495. 


\end{thebibliography}
\end{document}